\documentclass[11pt, preprint, secthm]{elsart}
\usepackage[T1]{fontenc}
\usepackage{amsmath,amssymb}
\usepackage{graphicx}
\usepackage{todonotes}
\graphicspath{ {images/} }

\renewcommand{\paragraph}[1]{\noindent{}{\bf #1}}

\newcommand{\RR}{{\mathcal R}}

\newcommand{\NN}{\mathbb{N}}
\newcommand{\R}{\mathbb{R}}

\newcommand{\II}{\mathbb{I}}

\DeclareMathOperator{\im}{Im}

\newcommand{\ma}{{\ensuremath{\mathfrak{a}}}} 
 
\newcommand{\mb}{{\ensuremath{\mathfrak{b}}}}
\newcommand{\md}{{\ensuremath{\mathfrak{d}}}}

\newcommand{\ms}{{\ensuremath{\mathfrak{s}}}}

\newcommand{\bg}{{\ensuremath{\mathbf{g}}}} 
\newcommand{\bh}{{\ensuremath{\mathbf{h}}}}

\newcommand{\cB}{{\ensuremath{\mathcal{B}}}}

\newcommand{\cG}{{\ensuremath{\mathcal{G}}}}
\newcommand{\cS}{{\ensuremath{\mathcal{S}}}}
\newcommand{\cM}{{\ensuremath{\mathcal{M}}}}

\newcommand{\cQ}{{\ensuremath{\mathcal{Q}}}} 
\newcommand{\cR}{{\ensuremath{\mathcal{R}}}}

\newcommand{\cK}{{\ensuremath{\mathcal{K}}}}

\newcommand{\cE}{{\ensuremath{\mathcal{E}}}} 
\newcommand{\cF}{{\ensuremath{\mathcal{F}}}}
\newcommand{\cH}{{\ensuremath{\mathcal{H}}}} 
\newcommand{\cU}{{\ensuremath{\mathcal{U}}}}
\newcommand{\cV}{{\ensuremath{\mathcal{V}}}}  
\newcommand{\cY}{{\ensuremath{\mathcal{Y}}}}   
\newcommand{\cZ}{{\ensuremath{\mathcal{Z}}}}   
\newcommand{\spanset}{\mathrm{span}}
\newcommand{\img}{\mathrm{Im}}
\newcommand{\Syz}{\mathrm{Syz}}

\def\ub{\mathbf{u}}

\def\cross{\mathfrak{c}}
\newcommand{\syz}{\mathrm{Syz}}

\renewcommand\ge\geqslant
\renewcommand\geq\geqslant
\renewcommand\le\leqslant
\renewcommand\leq\leqslant

\newtheorem{theorem}[thm]{Theorem}

\newenvironment{proof}{\par\noindent{\em Proof.}}{\ \hfill $\Box$\\}

\newtheorem{remark}[thm]{Remark}

\newtheorem{df}[thm]{Definition}
\newtheorem{lm}[thm]{Lemma}

\newtheorem{condition}[thm]{Condition}

\begin{document}
\begin{frontmatter}
\title{$G^1$-smooth splines on quad meshes with 4-split macro-patch elements}

\author[inria]{Ahmed Blidia}\ead{Ahmed.Blidia@inria.fr}\ 
\author[inria]{Bernard Mourrain}\ead{Bernard.Mourrain@inria.fr}\ 
\author[ricam]{Nelly Villamizar}\ead{n.y.villamizar@swansea.ac.uk}

\address[inria]{UCA, Inria Sophia Antipolis M\'editerran\'ee, \textsc{aromath}, Sophia Antipolis, France}
\address[ricam]{Swansea University, Swansea, UK}

\begin{abstract}
  We analyze the space of differentiable functions on a quad-mesh
  $\cM$, which are composed of 4-split spline macro-patch elements on
  each quadrangular face.  We describe explicit transition maps across
  shared edges, that satisfy conditions which ensure that the space of
  differentiable functions is ample on a quad-mesh of arbitrary
  topology.  These transition maps define a finite dimensional vector
  space of $G^{1}$ spline functions of bi-degree $\le (k,k)$ on each
  quadrangular face of $\cM$.  We determine the dimension of this space
  of $G^{1}$ spline functions for $k$ big enough and provide explicit
  constructions of basis functions attached respectively to vertices,
  edges and faces.  This construction requires the analysis of the
  module of syzygies of univariate b-spline functions with b-spline
  function coefficients. New results on their generators and
  dimensions are provided.  Examples of bases of $G^{1}$ splines of
  small degree for simple topological surfaces are detailed and
  illustrated by parametric surface constructions.
\end{abstract}

\end{frontmatter}
 
\section{Introduction}

Quadrangular b-spline surfaces are ubiquitous in geometric modeling.
They are represented as tensor products of univariate b-spline
functions. Many of the properties of univariate b-splines extend
naturally to this tensor representation. They are well suited to describe
parts of shapes, with features organized along two different
directions, as this is often the case for manufactured objects.
However, the complete description of a shape by tensor product b-spline patches
may require to intersect and trim them, resulting in a geometric model,
which is inaccurate or difficult to manipulate or to deform.

To circumvent these difficulties, one can consider geometric models
composed of quadrangular patches, glued together in a smooth way along
their common boundary.  Continuity constraints on the tangent planes
(or on higher osculating spaces) are imposed along the share edges. In
this way, smooth surfaces can be generated from quadrilateral meshes
by gluing several simple parametric surfaces.  By specifying the
topology of a quad(rilateral) mesh $\cM$ and the geometric continuity
along the shared edges via transition maps, we obtain a (vector) space
of smooth b-spline functions on $\cM$.

Our objective is to analyze this set of smooth b-spline functions on a
quad mesh $\cM$ of arbitrary topology.  In particular, we want to
determine the dimension and a basis of the space of smooth functions
composed of tensor product b-spline functions of bounded degree.  By
determining bases of these spaces, we can represent all the smooth
parametric surfaces which satisfy the geometric continuity conditions
on $\cM$.  Any such surface is described by its control points in this
basis, which are the coefficients in the basis of the differentiable
functions used in the parametrization.

The construction of basis functions of a spline space has several
applications.  For visualization purposes, smooth deformations of
these models can be obtained simply by changing their
coefficients in the basis, while keeping satisfied the regularity constraints
along the edges of the quad mesh.  Fitting problems for constructing
smooth models that approximate point sets or satisfy geometric
constraints can be transformed into least square problems on the
coefficient vector of a parametric model and solved by standard
linear algebra tools.  Knowing a basis of the space of smooth spline
functions of bounded degree on a quad mesh can also be useful in
Isogeometric Analysis. In this methodology, the basis functions are
used to describe the geometry and to approximate the solutions of
partial differential equations on the geometry. The explicit knowledge
of a basis allows to apply Galerkin type methods, which project the
solution onto the space spanned by the basis functions.

In the last decades, several works have been focusing on the construction of $G^{1}$ surfaces,  including
\cite{catmull_recursively_1978}, 
\cite{peters_biquartic_1995},
\cite{loop_smooth_1994},
\cite{reif_biquadratic_1995},
\cite{prautzsch_freeform_1997},
\cite{ying_simple_2004},
\cite{gu_manifold_2006},
\cite{he_manifold_2006},
\cite{fan_smooth_2008},
\cite{hahmann_bicubic_2008},
\cite{peters_complexity_2010},
\cite{beccari_rags:_2014},
\cite{bonneau_flexible_2014}. 
Some of these constructions use tensor product b-spline elements.
In \cite{peters_biquartic_1995}, biquartic b-spline elements are used on
a quad mesh obtained by middle point refinement of a general mesh.
These elements involve 25 control coefficients.
In \cite{reif_biquadratic_1995},
biquadratic b-spline elements are used on a semi-regular quad mesh obtained by three levels of mid-point refinements. These correspond to 16-split macro-patches, which involve 81 control points.
In \cite{peters_patching_2000}, 
 bicubic b-spline elements with 3 nodes, corresponding to a 16-split of the parameter domain are used. 
The macro-patch elements involve 169 control coefficients.
In \cite{lin_adaptive_2007},
biquintic  polynomial elements are used for solving a fitting problem. They involve 36 control coefficients. Normal vectors are extracted from the data of the fitting problem to specify the $G^{1}$  constraints.
In \cite{shi_practical_2004}
biquintic 5-split b-spline elements are involved. They are represented by 100 control coefficients or more.
In \cite{fan_smooth_2008}, bicubic 9-split b-spline elements are involved. They are represented by 
100 control coefficients.
In \cite{peters_complexity_2010}, it is shown that bicubic $G^{1}$ splines with linear transition maps
requires at least a 9-split of the parameter domains.
In \cite{hahmann_bicubic_2008}, bicubic 4-split macro-patch elements are used.
They are represented by 36 control coefficients.  The construction does not apply for general quad meshes.
In \cite{bonneau_flexible_2014},  biquartic 4-split macro-elements are used. They involve 81 control coefficients. The construction applies for general quad meshes and is used to solve the interpolation problem of boundary curves.
In these constructions, symmetric geometric continuity constraints are used at the vertices of the mesh.

Much less work has been developed on the dimension analysis.
In \cite{mourrain_dimension_2016}, a dimension formula and explicit basis constructions are given for polynomial patches of degree $\ge 4$ over a mesh with triangular or quadrangular cells.
In \cite{bercovier_smooth_2014}, a similar result is obtained for the space of $G^{1}$ splines of bi-degree $\ge (4,4)$ for rectangular decompositions of planar domains.
The construction of basis functions for spaces of $C^1$ geometrically continuous functions restricted to two-patch domains, has been considered in \cite{Kapl_2015}. In \cite{Sangalli_2016}, the approximation properties of the aforementioned spaces are explored, including constructions over multi-patch geometries motivated by applications in isogeometric analysis.

In this paper we analyze the space of $G^{1}$ splines on a general
quad mesh $\cM$, with 4-split macro-patch elements of bounded
bi-degree.  We describe explicit transition maps across shared edges,
that satisfy conditions which ensure that the space of differentiable
functions is ample on the quad mesh $\cM$ of arbitrary topology.
These transition maps define a finite dimensional vector space of
$G^{1}$ b-spline functions of bi-degree $\le (k,k)$ on each
quadrangular face of $\cM$.  We determine the dimension of this space
for $k$ big enough and provide explicit constructions of basis
functions attached respectively to vertices, edges and faces.  This
construction requires the analysis of the module of syzygies of
univariate b-spline functions with b-spline coefficients. New results
on their generators and dimensions are provided.  This yields a new
construction of smooth splines on quad meshes of arbitrary topology,
with macro-patch elements of low degree.

Examples of bases of $G^{1}$ splines of small degree
for simple topological surfaces are detailed and illustrated by
parametric surface constructions.

The techniques developed in this paper for the study of geometrically
smooth splines rely on the analysis of the syzygies of the glueing
data, similarly to the approach used in \cite{mourrain_dimension_2016}
for polynomial patches.  However an important difference is that we
consider here syzygies of spline functions with spline
coefficients. The classical properties of syzygy modules over the ring
of polynomials used in \cite{mourrain_dimension_2016} do not apply 
to spline functions. New results on syzygy modules over the ring of
piecewise polynomial functions (with one node and prescribed
regularity) are presented in Sections \ref{sec:relation_syz},
\ref{sec:syz_basis}. In particular, Proposition \ref{demi} describes a
family of gluing data with additional degrees of freedom, providing
new possibilities to construct ample spaces of spline functions in
low degree.  The necessary notation and constrains from the case of
polynomial patches are used in the course of the paper.  But,
properties of Taylor maps exploited in \cite {mourrain_dimension_2016}
are incoherent in our context.  In our setting, the construction of
the spline space requires to extend the results on these Taylor maps
to the context of macro-patches.  Sections \ref{sec:sepvert},
\ref{sec:polybasis}, \ref{sec:basis_edge} present an alternative
analysis adapted to our needs. Exploiting these properties, vertex
basis functions and face basis functions can then be constructed in
the same way as in the polynomial case.


The paper is organized as follows. The next section recalls the notions of topological surface
$\cM$, differentiable functions on $\cM$  and smooth spline functions on $\cM$.
In Section 3,  we detail the constraints on the
transition maps to have an ample space of differentiable functions and provide
explicit constructions.
In Section 4, 5, 6,  we analyze the space of smooth spline functions around respectively an edge, a vertex and a face and describe basis functions attached to these elements. 
In Section 7, we give the dimension formula for the space of spline functions of bi-degree $\le (k,k)$ over a general quad mesh $\cM$ and describe a basis.  Finally, in Section
7, we give examples of such smooth spline spaces.
\section{Definitions and basic properties}
In this section, we define and describe the objects we need to analyze the spline spaces on a quad mesh.
\subsection{Topological surface}
\begin{df}
A \textbf{topological surface} $\cM$ is given by
\begin{itemize}
\item a collection $\cM_2$ of polygons (also called faces of $\cM$) in the plane that are pairwise disjoint,

\item a collection of homeomorphisms $\phi _{i,j}: \tau_i\mapsto \tau_j$ between polygonal edges from different polygons $\sigma _i$ and $\sigma _j$ of $\cM_{2}$,
\end{itemize}
where a polygonal edge can be glued with at most one other polygonal edge, and it cannot be glued with itself.  The shared edges (resp. the points of the shared edges) are identified with their image by the corresponding homeomorphism. The collection of edges (resp. vertices) is denoted $\cM_{1}$ (resp. $\cM_{0}$).
\end{df}

For a vertex $\gamma\in \cM_{0}$, we denote by $\cM_{\gamma}$ the submesh of $\cM$ composed of the faces which are adjacent to $\gamma$.
For an edge $\tau \in \cM_{1}$, we denote by $\cM_{\tau}$ the submesh of $\cM$ composed of the faces which are adjacent to the interior of $\tau$.
\begin{df}[Gluing data]
For a topological surface $\mathcal{M}$, a gluing structure associated to 
	$\mathcal{M}$ consists of the following:
	\begin{itemize}
	\item for each edge $\tau\in \mathcal{M}_{1}$ of a cell $\sigma$, an open set 
		$U_{\tau,\sigma}$ of $\mathbb{R}^2$  containing $\tau$;
 	\item for each edge $\tau\in\mathcal{M}_{1}$ shared by two polygons
 		${\sigma_i},\sigma_j\in \mathcal{M}_{2}$, a C$\,^{1}$-diffeomorphism 
 		called the {\em transition map} 
 		$\phi_{\sigma_j,\sigma_i}\colon U_{\tau,\sigma_i}\rightarrow U_{\tau,\sigma_j}$ 
 		between the open sets $U_{\tau,\sigma_i}$ and $U_{\tau,\sigma_j}$, 
 		and also its 	correspondent inverse map $\phi_{\sigma_i,\sigma_j}$;
	\end{itemize}

\end{df}
Let $\tau$ be an edge shared by two polygons
$\sigma_{0},\sigma_{1} \in \mathcal{M}_{2}$ , $\tau=\tau_{0}$ in $\sigma_{0}$, $\tau=\tau_{1}$ in $\sigma_{1}$ respectively and let $\gamma=(\gamma_{0},\gamma_{1})$ be a vertex of $\tau$ corresponding to $\gamma_{0}$ in $\sigma_{0}$ and to $\gamma_{1}$ in $\sigma_{1}$. We denote by $\tau_{1}'$
(resp. $\tau'_{0}$) the second edge of $\sigma_{1}$
(resp. $\sigma_{0}$) through $\gamma_{1}$ (resp. $\gamma_{0}$).
We associate to $\sigma_{1}$ and $\sigma_{0}$ two coordinate systems
$(u_{1},v_{1})$ and $(u_{0},v_{0})$ such
that $\gamma_{1}=(0,0)$, $\tau_{1}=\{(u_{1},0), u_{1}\in [0,1]\}$, $\tau'_{1}=\{(0,v_{1}), v_{1}\in [0,1]\}$
and
$\gamma_{0}=(0,0)$, $\tau_{0}=\{(0,v_{0}), v_{0}\in [0,1]\}$, $\tau'_{0}=\{(u_{0},0), u_{0}\in [0,1]\}$, see Figure \ref{fig:gluing}.
\begin{figure}
\begin{center}
  \includegraphics[scale=1.3]{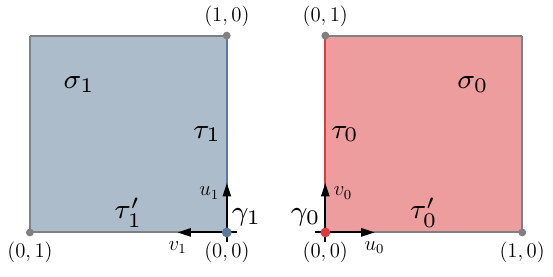}
\end{center}
\caption{Given an edge $\tau$ of a topological surface $\cM$ that is shared by two polygons $\sigma_0,\sigma_1\in\cM$, we associated a different coordinate system to each of these two faces and consider $\tau$ as the pair of edges $\tau_0$ and $\tau_1$ in $\sigma_0$ and $\sigma_1$, respectively.}\label{fig:gluing}
\end{figure}
Using the Taylor expansion at $(0,0)$, a transition map from $U_{\tau,\sigma_1}$ to $U_{\tau,\sigma_0}$ is then of the form 
\begin{equation} \label{eq:transmap}
	\phi_{\sigma_0,\sigma_1}\colon (u_{1},v_{1}) \longrightarrow  (u_{0},v_{0})=
	\begin{pmatrix}
		v_{1} \,\mathfrak{b}_{\tau,\gamma}(u_{1}) + v_{1}^{2} \rho_{1}(u_{1},v_{1})\\
		u_{1}+v_{1}\,\mathfrak{a}_{\tau,\gamma}(u_{1}) + v_{1}^{2} \rho_{2}(u_{1},v_{1})
	\end{pmatrix}
\end{equation}
where $\mathfrak{a}_{\tau,\gamma}(u_{1}), \mathfrak{b}_{\tau,\gamma}(u_{1}),\rho_{1}(u_{1},v_{1}),\rho_{2}(u_{1},v_{1})$ are $C^{1}$ functions.
We will refer to it as the canonical form of the transition map
$\phi_{0,1}$ at $\gamma$ along $\tau$. The functions
$[\mathfrak{a}_{\tau,\gamma},\mathfrak{b}_{\tau,\gamma}]$ are called the {\em gluing data} at $\gamma$
along $\tau$ on $\sigma_{1}$.
\begin{df}\label{def:crossing}
	An edge $\tau\in\cM$ which contains the vertex $\gamma \in \cM$ 
 	is called a {{\em crossing edge at $\gamma$}} if
	$\ma_{\tau,\gamma}(0)=0$ where $[\ma_{\tau,\gamma},\mb_{\tau,\gamma}]$ is the gluing
 	data at $\gamma$ along $\tau$.
 	We define $\cross_{\tau}(\gamma)=1$ if $\tau$ is a 
 	crossing edge at $\gamma$ and $\cross_{\tau}(\gamma)=0$ otherwise.
 	By convention, $\cross_{\tau}(\gamma)=0$ for a boundary edge.
 	If $\gamma\in\cM_0$ is an interior vertex where all adjacent 
 	edges are crossing edges at 
 	$\gamma$, then it is called a {\em crossing vertex}.
 	Similarly, we define $\cross_{+}(\gamma)=1$ if $\gamma$ is a 
 	crossing vertex and $\cross_{+}(\gamma)=0$ otherwise.
\end{df}

\subsection{Differentiable functions}
We define now the differentiable functions on $\cM$ and the spline functions on $\cM$.
\begin{df}[Differentiable functions]
A differentiable function on a topological surface $\cM$ is a collection $f=(f_{\sigma})_{\sigma\in \mathcal{M}_2}$ of differentiable functions such that 
for each two faces $\sigma _0$ and $\sigma_1$ sharing an edge $\tau $ with $\phi_{0,1}$ as transition map, the two functions $f_{\sigma _1}$ and $f_{\sigma _0}\circ \phi_{0,1}$  have the same Taylor expansion of order 1. The function $f_{\sigma}$ is called the restriction of $f$ on the face $\sigma$.
\end{df}

This leads to the following two relations for each $u_1\in [0,1]$:
\begin{eqnarray}
  f_{1}(u_1,0)&=&f_{0}(0,u_1)\label{eq:edgecond0}\\
\frac{\partial f_{1}}{\partial v_1}(u_1,0)&=&\mathfrak{b}_{\tau,\gamma}(u_1)\frac{\partial f_0}{\partial u_0} (0,u_1)+\mathfrak{a}_{\tau,\gamma }(u_1)\frac{\partial f_0}{\partial v_0}(0,u_1)
\label{eq:edgecond}
\end{eqnarray}
where$f_1=f_{\sigma_1}$,  $f_0=f_{\sigma _0}$ are the restrictions of $f$ on the faces $\sigma_{0}$, $\sigma_{1}$.

For $r\in \NN$, let $\cU^{r}= \cS^{r}([0,\frac{1}{2},1])$ be the space of piecewise univariate polynomial functions (or splines) on the subdivision $[0,\frac{1}{2},1]$, which are of class $C^{r}$. We denote by $\cU^{r}_{k}$ the spline functions in $\cU^{r}$ whose polynomial pieces are of degree $\le k$.
We denote by $\RR[u]$ the ring of polynomials in one variable $u$, with coefficients in $\RR$.

Let $\cR^r(\sigma)$ be the space of spline functions of regularity $r$ in each parameter over the $4$-split subdivision of the quadrangle $\sigma$ (see Figure \ref{fig:4-split}), that is, the tensor product of $\cU^{r}$ with itself. 
\begin{figure}
\begin{center}
  \includegraphics[scale=1.75]{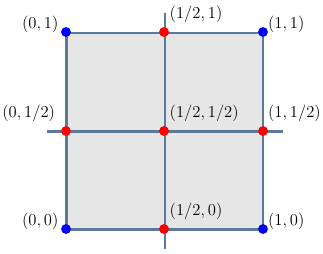}
\end{center}
\caption{4-split of the parameter domain}\label{fig:4-split}
\end{figure}

For $k\in \NN$, the space of b-spline functions of degree $\le k$ in each variable, that is of bi-degree $\le (k,k)$ is denoted $\cR^{r}_{k}(\sigma)$.
A function $f_{\sigma} \in \cR_{k}^{r}(\sigma)$ is represented in the b-spline basis of $\sigma$ as
$$ 
f_{\sigma} := \sum_{0 \le i,j\le m} c_{i,j}^{\sigma}(f_{\sigma}) N_{i}(u_{\sigma}) N_{j}(v_{\sigma}),
$$
where $c_{i,j}^{\sigma}(f_{\sigma})\in \RR$  and $N_{0},\ldots, N_{m}$ are the classical b-spline basis functions of $\cU_{k}^{r}$ with $m= 2 k - r$. The dimension of  $\cR_{k}^{r}(\sigma)$ is $(m+1)^{2}= (2 k-r+1)^{2}$.

The geometric continuous spline functions on $\cM$ are the differentiable functions $f$ on $\cM$, where each component $f_{\sigma}$ on a face $\sigma\in \cM_{2}$ is in $\cR^{r}(\sigma)$. We denote this spline space by $\cS^{1,r}(\cM)$. The set of  splines $f\in \cS^{1,r}(\cM)$ with $f_{\sigma}\in \cR^{r}_{k}(\sigma)$ is denoted $\cS^{1,r}_{k}(\cM)$.

\subsection{Taylor maps}
\label{sec:tay}
An important tool that we are going to use intensively is the Taylor
map associated to a vertex or to an edge of $\cM$.
 For each face $\sigma$  the space of spline functions over a subdivision onto 4 parts as in the figure above will be denoted $\mathcal{R}^r(\sigma)$.
Let $\gamma\in \cM_{0}$ be a vertex on a face $\sigma\in
\cM_{2}$ belonging to two edges $\tau, \tau' \in \cM_{1}$ of
$\sigma$. We define the {\em ring of $\gamma$ on $\sigma$} by 
 $\cR^{\sigma}(\gamma)= \cR(\sigma)/(\ell_{\tau}^{2}, \ell_{\tau'}^{2})$
where $(\ell_{\tau}^{2}, \ell_{\tau'}^{2})$ is the ideal generated by 
the squares of $\ell_\tau$ and $\ell_{\tau'}$, the equations $\ell_{\tau}(u,v)=0$ and $\ell_{\tau'}(u,v)=0$
are respectively the equations of $\tau$ and $\tau'$ in $\cR^r(\sigma)=\mathcal{S}^r$.

The {\em Taylor expansion at $\gamma$ on $\sigma$} is the map
\[
	T_{\gamma}^{\sigma}: f \in \cR^r(\sigma) \mapsto f \mod 
	(\ell_{\tau}^{2}, \ell_{\tau'}^{2}) \text{\; in\; } \cR^{\sigma}(\gamma).
\]
Choosing an adapted basis of $\cR^{\sigma}(\gamma)$, one can define $T_{\gamma}^{\sigma}$ by 
\[
	T_{\gamma}^{\sigma}(f) = \bigl[\, f(\gamma),\, \partial_{u} f(\gamma), \, 
	\partial_{v} 		f(\gamma),\, \partial_{u}\partial_{v} f(\gamma)\, \bigr]. 
\]
The map $T_{\gamma}^{\sigma}$ can also be defined in  another basis
of $\cR^{\sigma}(\gamma)$ in terms of the b-spline coefficients by
\[ 
	T_{\gamma}^{\sigma}(f) = \bigl[\, c^{\sigma}_{0,0}(f),\,  c^{\sigma}_{1,0}(f), \, c^{\sigma}_{0,1}(f),\,  
	c^{\sigma}_{1,1}(f) \, \bigr]
\]
where $c_{0,0}, c_{1,0}, c_{0,1}, c_{1,1}$ are the first b-spline coefficients associated to $f$ on $\sigma$ at $\gamma=(0,0)$. 

We define the Taylor map  $T_\gamma$ on all the faces $\sigma$ that contain $\gamma$, 
\[
	T_{\gamma}: f=(f_{\sigma})\in \oplus_{\sigma}\cR^r(\sigma)
	\rightarrow (T_{\gamma}^{\sigma}(f_{\sigma}))\in
	\oplus_{\sigma\supset\gamma} \cR^{\sigma}(\gamma).
\]
Similarly, we define $T$ as the Taylor map at all the vertices on all the faces of  $\cM$.
\smallskip

If $\tau\in \cM_{1}$ is the edge of the face $\sigma(u_\sigma,v_\sigma) \in \cM_{2}$ associated to $v_{\sigma}=0$, we define the {\em restriction along $\tau$ on $\sigma$} as 
\begin{eqnarray*}
  D_{\tau}^{\sigma}: \cR_{k}^r(\sigma) & \rightarrow & \cR_{k}^r(\sigma) \\ 
  f = \sum_{0\le i,j \le m} c_{i,j}^{\sigma}(f) N_{i}(u_{\sigma}) N_{j}(v_{\sigma}) 
& \mapsto & \sum_{0\le i \le m, 0\le j \le 1} c_{i,j}^{\sigma}(f) N_{i}(u_{\sigma}) N_{j}(v_{\sigma}). 
\end{eqnarray*}
The restrictions along the edges $v_{\sigma}=1$, $u_{\sigma}=0$, $u_{\sigma}=1$ are defined similarly by symmetry. By convention if $\tau$ is not an edge of $\sigma$, $D_{\tau}^{\sigma}=0$. 

For a face $\sigma\in \cM_{2}$, we define the {\em restriction along the edges of $\sigma$} as 
\begin{eqnarray*}
  D^{\sigma}\colon \cR_{k}^r(\sigma) & \rightarrow & \cR_{k}^r(\sigma) \\
  f = \sum_{0\le i,\, j \le m} c_{i,j}^{\sigma}(f) N_{i}(u_{\sigma}) N_{j}(v_{\sigma}) 
& \mapsto & \sum_{\substack{i> 1,\text{ or}\\i<m-1,\, j>1,\\ \text{ or}\  j< m-1}} c_{i,j}^{\sigma}(f) N_{i}(u_{\sigma}) N_{j}(v_{\sigma}).
\end{eqnarray*}
The edge restriction map along all edges of $\cM$ is given by 
\[
	D: f=(f_{\sigma})\in \oplus_{\sigma}\cR_{k}^r(\sigma) 
	\rightarrow (D^{\sigma}(f_{\sigma}))\in 
        \oplus_{\sigma}\cR_{k}^r(\sigma). 
\]

\section{Transition maps}\label{sectran}

The spline space on the mesh $\cM$ is constructed using the transition maps associated to the edges shared by pair of polygons in $\cM$. The transition map accross an edge $\tau$ is given by formula \eqref{eq:transmap}, where
$\ma(u)=\frac{a(u)}{c(u)}, \mb(u)=\frac{b(u)}{c(u)}$ and $[a(u), b(u), c(u)]$ is a triple of functions, called gluing data. In the following, the transition maps will be defined from spline functions  in $\cU^{r}_{l}$, of class $C^{r}$  and degree $l$, with nodes $0, \frac{1}{2}, 1$ for the gluing data.
We assume that the dimension of $\cU^{r}_{l}$ is bigger than $4$, that is, $2l+1-r\geq4$ and $r\geq 0$ so that $l\geq \frac{1}{2}(3 + r)$, which implies that $l\geq 2$. 

We denote by $\md_{0}(u), \md_{1}(u)\in \cU^{r}_{l}$ two spline functions such that
$\md_{0}(0)=1$, $\md_{0} (1)=0$, $\md_{1} (0) =0$, $\md_{1}(1)=1$
and $\md'_{0}(0)=\md'_{0} (1)=\md'_{1} (0) =\md'_{1}(1)=0$. 
We can take, for instance, 
\begin{align}
  \md_{0}(u) &= N_{0}(u)+ N_{1}(u) \label{eq:delta01}\\
  \md_{1}(u) &= N_{m-1}(u)+ N_{m}(u) \nonumber
\end{align}
where $m=2l-r$.
For $l=2$, $r=1$, these functions are 
\begin{align*}
\md_{0}(u) &=\begin{cases}
               1 - 2 u^{2}  & \quad 0\leq u\leq \frac{1}{2}\\ 
               2 (1-u)^{2}  & \quad \frac{1}{2}\leq u \leq 1
             \end{cases}\\
\md_{1}(u) &=\begin{cases}
               2 u^{2} &\quad 0\leq u \leq \frac{1}{2}\\ 
             1 - 2\,(1-u)^{2} & \quad \frac{1}{2}\leq u \leq 1.
             \end{cases}
\end{align*}
For $l=2$, $r=0$, these functions are 
\begin{align*}
  \md_{0}(u) &=\begin{cases}
               1 - 4 u^{2} &  \quad 0\leq u\leq \frac{1}{2}\\ 
               0 & \quad \frac{1}{2}\leq u \leq 1
              \end{cases}\\   
\md_{1}(u) &=\begin{cases}
               0 & \quad 0\leq u \leq \frac{1}{2}\\ 
             1 - 4\,(1-u)^{2} & \quad \frac{1}{2}\leq u \leq 1.
             \end{cases}
\end{align*}
To ensure that the space of spline functions is sufficiently ample (i.e., it contains enough regular functions, see \cite[Definition 2.5]{mourrain_dimension_2016}), we impose compatibility conditions. 

First around an interior vertex $\gamma \in \cM_{0}$, 
which is common to faces $\sigma_{1}, \ldots,
\sigma_{F}$ glued cyclically around $\gamma$, along the edges $\tau_{i}=\sigma_{i}\cap\sigma_{i-1}$ for $i=2,\ldots, F+1$ (with 
$\sigma_{F+1}=\sigma_{1}$), we impose the condition:
$J_{\gamma}(\phi_{1,2})  \circ \cdots \circ  J_{\gamma}(\phi_{N-1,N}) = \II_2$
where $J_{\gamma}$ is the jet or Taylor expansion of order $1$ at $\gamma$. 
%
%
It translates into the following condition (see \cite{mourrain_dimension_2016}):
\begin{condition}\label{cond:comp1}
If  $\gamma\in \cM_{0}$ is an interior vertex and belongs to the faces $\sigma_{1},\ldots,\sigma_{F}$ that are glued cyclically around $\gamma$, then the gluing data $[\ma_{i}, \mb_{i}]$ at $\gamma$ on  the edges $\tau_{i}$ between $\sigma_{i-1}$ and $\sigma_{i}$ satisfies 
	\begin{equation}\label{eq:gc1vertex}
		\prod_{i=1}^F \left(
                  \begin{array}{cc}
                    0 & 1 \\
                    \mb_i(0) & \ma_i(0)

                  \end{array}
		\right) =\left(
		\begin{array}{cc} 1 & 0 \\ 0 & 1 \end{array} \right).
	\end{equation}
\end{condition}
This gives algebraic restrictions on the values $\ma_i(0)$, $\mb_i(0)$.

In addition to Condition~\ref{cond:comp1}, we also consider the following condition around a crossing vertex:
\begin{condition}\label{cond:comp2}
  If the vertex $\gamma$ is a crossing vertex with $4$ edges $\tau_{1},\ldots,\tau_{4}$,
  the gluing data $[\ma_{i},\mb_{i}]$ $i=1 \ldots 4$ on these edges at
  $\gamma$ satisfy
	\begin{align} 
		\ma'_1(0)+\frac{\mb'_4(0)}{\mb_4(0)} &= -\mb_1(0) 
		\left(\ma'_3(0)+\frac{\mb'_2(0)}{\mb_2(0)}\right),\label{eq:ddv1}\\ 
				\ma'_2(0)+\frac{\mb'_1(0)}{\mb_1(0)} &= -\mb_2(0) 
		\left(\ma'_4(0)+\frac{\mb'_3(0)}{\mb_3(0)}\right). \label{eq:ddv2}
	\end{align}
\end{condition}
Let us notice that we can write the previous conditions on the gluing data (which in our setting is given by spline functions) as in \cite{mourrain_dimension_2016} since they depend on the value of the functions defining the gluing data and not on the particular type of functions.  
The conditions (\ref{eq:ddv1}) and (\ref{eq:ddv2}) were introduced in \cite{mourrain_dimension_2016} in the context of gluing data defined from polynomial functions, they generalize the conditions of \cite{peters_complexity_2010}, where $\mb_{i}(0)=-1$. The conditions  come from the relations between the derivatives and the cross-derivatives of the face functions across the edges at a crossing vertex.

An additional condition of topological nature is also considered in \cite{mourrain_dimension_2016}. It ensures that the glued faces around a vertex $\gamma$ are equivalent to sectors around a point in the plane, via the reparameterization maps. We will not need it hereafter.

To define transition maps which satisfy these conditions, we first compute the values of the transition functions $\ma_{\tau}, \mb_{\tau}$ of an edge $\tau$ at its end points and then interpolate the values: 
\begin{enumerate}
	\item For all the vertices $\gamma\in \cM_{0}$ and for all the edges
  	$\tau_{1},\ldots, \tau_{F}$ of $\cM_{1}$ that contain $\gamma$, choose vectors
   	$\ub_{1},\ldots,\ub_{F}\in \RR^{2}$ such that the cones in $\RR^2$ generated by
   	$\ub_{i},\ub_{i+1}$ form a fan in $\RR^{2}$  and such that the
   	union of these cones is $\RR^{2}$ when $\gamma$ is an interior vertex.
    The vector $\ub_i$ is associated to the edge $\tau_i$, so that the sectors $\ub_{i-1},\ub_{i}$ and $\ub_{i},\ub_{i+1}$ define the gluing across the edge $\tau_i$ at $\gamma$. 
  
   	The transition map $\phi_{i-1,i}$ at 
   	$\gamma=(0,0)$ on the edge $\tau_{i}$ is constructed as:
		\[
			J_{(0,0)}(\phi_{i-1,i})^{t} = S \circ
			[\ub_{i},\ub_{i+1}]^{-1} \circ  [\ub_{i-1},\ub_{i}] \circ S
			=\left[ \begin{array}{cc}0&\mb_{{i}}(0)\\ 
			1& \ma_{{i}}(0)\end{array}\right]
		\]
		where $S=\begin{bmatrix}
0 & 1\\
1 & 0
\end{bmatrix}$,
		$[\ub_{i},\ub_{j}]$ is the matrix which columns are the vectors
		$\ub_{i}$ and $\ub_{j}$, and $|\ub_{i},\ub_{j}|$ is the determinant of the
		vectors $\ub_{i},\ub_{j}$.	Thus, 
		\begin{equation}\label{eq:sector}
			\ma_{{i}}(0) = \frac{|\ub_{i+1},\ub_{i-1}|}{|\ub_{i+1},\ub_{i}|},\;
			\mb_{{i}}(0) = - \frac{|\ub_{i},\ub_{i-1}|}{|\ub_{i+1},\ub_{i}|},
		\end{equation}
		so that 
		$\ub_{i-1}= \ma_{{i}}(0) \ub_{i} + \mb_{{i}}(0) \ub_{i+1}$. This implies that Condition \ref{cond:comp1} is satisfied.
	\item For all the shared edges $\tau\in \cM_{1}$, we define the functions	
		$\ma_{\tau}=\frac{a_{\tau}}{c_{\tau}},\mb_{\tau}=\frac{b_{\tau}}{c_{\tau}}$ on the edges $\tau$ by interpolation as follows.
		Assume that the edge $\tau$ is associated to the vectors $\ub^{0}$ and $\ub^{1}$, respectively at the end point $\gamma$ and $\gamma'$ corresponding to the parameters $u=0$ and $u=1$.  
		Let $\ub_{-}^{s},\ub_{+}^{s}\in \RR^{2}$, $s=0,1$ be the vectors which define respectively the previous and next sectors adjacent to $u_i^{s}$ at the point $\gamma$ and $\gamma'$, see Figure ~\ref{fig:edge_int}. We define the gluing data so that it interpolates the corresponding value \eqref{eq:sector} at $u=0$ and $u=1$ as:
\begin{align}
a_\tau(u) &=\; \bigl|\ub_{+}^{0},\ub_{-}^{0}\bigr| \md_{0}(u) + \bigl|\ub_{+}^{1},\ub_{-}^{1}\bigr| \md_{1}(u)\nonumber\\
b_\tau(u) &=-\bigl|\ub^{0},\ub_{-}^{0}\bigr| \md_{0}(u) - \bigl|\ub^{1},\ub_{-}^{1}\bigr| \md_{1}(u)\label{eq:construct1}\\
c_\tau(u) &=\; \bigl|\ub_{+}^{0},\ub^{0}\bigr| \md_{0}(u) + \bigl|\ub_{+}^{1},\ub_{}^{1}\bigr| \md_{1}(u)\nonumber	
\end{align}
where $\md_{0}(u), \md_{1}(u)$ are two Hermite interpolation functions at $u=0$ and $u=1$.

\begin{figure}
\begin{center}
  \includegraphics[scale=0.8]{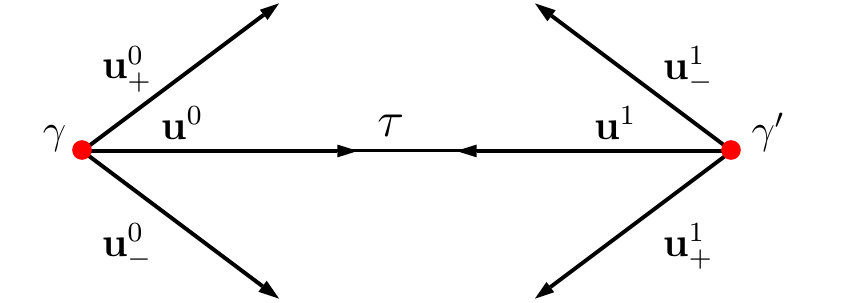}
\end{center}

\caption{The edge $\tau=(\gamma,\gamma')$ is associated to the vectors $\ub^0$ and $\ub^1$ at the points $\gamma$ and $\gamma'$, respectively.}\label{fig:edge_int}
\end{figure}
 
Since the derivatives of $a_\tau, b_\tau, c_\tau$ vanish at $u=0$ and $u=1$, the conditions (\ref{eq:ddv1}) and (\ref{eq:ddv2}) are automatically satisfied at an end point if it is a crossing vertex. 

Another possible construction, with a constant denominator $c_\tau(u)=1$ is:
\begin{align}
a_\tau(u) &=\; \frac{\bigl|\ub_{+}^{0},\ub_{-}^{0}\bigr|}{\bigl|\ub_{+}^{0},\ub_{}^{0}\bigr|} \md_{0}(u) - \frac{\bigl|\ub_{+}^{1},\ub_{-}^{1}\bigr|}{\bigl|\ub_{+}^{1},\ub_{}^{1}\bigr|} \md_{1}(u)\nonumber\\
b_\tau(u) &= - \frac{\bigl|\ub_{}^{0},\ub_{-}^{0}\bigr|}{\bigl|\ub_{+}^{0},\ub_{}^{0}\bigr|} \md_{0}(u) + \frac{\bigl|\ub_{}^{1},\ub_{-}^{1}\bigr|}{\bigl|\ub_{+}^{1},\ub_{}^{1}\bigr|} \md_{1}(u)\label{eq:construct2}\\
c_\tau(u) &=  1\nonumber
\end{align}
\end{enumerate}
The construction \eqref{eq:construct2} specializes to the symmetric gluing used for instance in \cite[\S 8.2]{hahn_geometric_1989}, \cite{hahmann_bicubic_2008}, \cite{bonneau_flexible_2014}:
\begin{align}
		 a_{\tau}&=\; \md_{0} (u) \,2 \cos {\frac{2\pi}{n_{0}}} 
		- \md_{1} (u)  \, 2 \cos \frac{2\pi}{n_{1}}\nonumber\\
		b_{\tau} &=\; -1 \label{eq:hahngluing}\\
		c_{\tau} &=\; 1\nonumber
\end{align} 
where $n_{0}$ (resp. $n_{1}$) is the number of edges at the vertex $\gamma_{0}$ 
(resp. $\gamma_{1}$).
It corresponds to a symmetric gluing, where the angle of two consecutive
edges at $\gamma_{i}$ is $\frac{2\pi}{n_{i}}$.
%

\section{Splines along an edge}\label{sec:4}

The space $\cS^{1,r}_k(\cM)$ of splines over the mesh $\cM$ can be splitted into three linearly independent components: $\cE_k$, $\cH_k$, $\cF_k$ (see Section \ref{sec:dim}) attached respectively to vertices, edges and faces. The objective of this section is to give a dimension formula for the component $\cE(\tau)_k$ attached to the  edge $\tau$ and an explicit base, where $\tau$ is an interior edge, shared by two faces $\sigma_{1}$, $\sigma_{2}\in \cM_{2}$. We denote by  $\cM_{\tau}$ the sub-mesh of $\cM$ composed of the two faces $\sigma_{1}, \sigma_{2}$.

An important step is to analyse the space $\syz^{r,r,r}_k(a,b,c)$ of Syzygies over the base ring $\cU^r$. The relation of this space with $\cE(\tau)_k$ and a basis of $\syz^{r,r,r}_k(a,b,c)$ are presented in Sections \ref{sec:relation_syz} and \ref{sec:syz_basis}.

Next in  Section \ref{sec:sepvert}, we study the effect, on $\cE(\tau)_k$,  of the Taylor map at the two end points of $\tau$ and we determine when they can be separated by the Taylor map. 
 
 The Section \ref{sec:polybasis} shows how to decompose the space $\cS ^{1,r}_k$ for the simple mesh $\cM_{\tau}$, using this Taylor maps at the end points of $\tau$. The same technique will be used to decompose the space $\cS ^{1,r}_k(\cM)$, for a general mesh $\cM$.

\subsection{Relation with Syzygies}\label{sec:relation_syz}


Given spline functions $a, b, c\in \cU ^s_l$ defining the gluing data accross the edge $\tau\in\cM$, and $(f_{1},f_{2})\in \cS_{k}^{1}(\cM_{\tau})$, 
from \eqref{eq:edgecond} we have that:
\[
A(u_1)a(u_1)+B(u_1)b(u_1)+C(u_1)c(u_1)=0
\]
where 
\begin{align*}
	A(u_1) &=\frac{\partial f_2}{\partial v_2}(0,u_1) \in \cU^{0}_{k-1},\\
	B(u_1) &=\frac{\partial f_2}{\partial u_2}(0,u_1) \in \cU^{1}_{k}, \\
	C(u_1) &=-\frac{\partial f_1}{\partial v_1}(u_1,0) \in \cU^{1}_{k}.
\end{align*}
These  are the conditions imposed by the transition map across $\tau$. According to such data, and if the topological surface $\mathcal{M}_\tau$ contains two faces with one transition map along the shared edge $\tau$, then any differentiable spline functions $f=(f_1,f_2)$ over $\cM_\tau$ of bi-degree $\le (k,k)$ is given by the formula:
\begin{align}
 f_1(u_1,v_1)=\bigl(N_{1}(v_1)+N_{0}&(v_{1})\bigr)\left (a_0+\int_0^{u_1}A(t)dt \right)\label{eq18}\\
 &- \frac{1}{2k}\, N_{1}(v_1) C(u_1)+ E_1(u_1,v_1)\nonumber
\end{align}
\begin{align}
f_2(u_2,v_2)=\bigl(N_{1}(u_2)+N_{0}&(u_{2})\bigr) \left(a_0+\int_0^{v_2}A(t)dt \right)\label{eq19}\\
&+ \frac{1}{2k}\, N_{1}(u_2) B(v_2)+ E_2(u_2,v_2),\nonumber
\end{align}
since $N_{0}(0)=1$,  $N_{1}(0)=0$, $N'_{0}(0)=-2 k$, and $N'_{1}(0)= 2k$.

Here $a_{0}\in \RR$, the functions  $E_i \in \ker D_{\tau}^{\sigma_{i}}$ for $i=0,1$, 
and $A,\,B,\, C$ are spline functions of degree at most $k-1, k, k$ and class $C^{0}, C^{1}, C^{1}$, respectively. 

For $r_{1},r_{2},r_{3},k\in \NN$ and $a,b,c\in \cU_l^s$, we denote 
\[
\Syz_{k}^{r_{1},r_{2},r_{3}}(a,b,c)=\bigl\{(A,B,C)\in \cU_{k-1}^{r_{1}}\times \cU^{r_{2}}_{k} \times \cU^{r_{3}}_
{k}\mid A\,a + B\, b + C\, c =0 \bigr\}.
\]
We denote this vector space simply by $\syz_{k}^{r_{1},r_{2},r_{3}}$ when $a,b,c$ are implicitly given.

By \eqref{eq18} and \eqref{eq19},  the splines in $\cS^{1}_{k}(\cM_{\tau})$ with a support along the edge $\tau $ are in the image of the map:
\begin{eqnarray}\label{eq:deftheta}
	\Theta_\tau : \RR\times \syz^{0,1,1}_{k} &\rightarrow & \cS^{r}_k(\cM_\tau) \\
 (a_0, (A,B,C)) & \mapsto& \biggl( \left (a_0+\int_0^{u_1}A(t)dt \right)\, N_{0}(v_{1})\nonumber\\ &&
          \ \ + \left(a_0+\int_0^{u_{1}}A(t)dt - \frac{1}{2k}\,      C(u_1) \right) \, N_{1}(u_1),\nonumber \\
   && \ \ N_{0}(u_{2}) \left(a_0+\int_0^{v_2}A(t)dt \right)\nonumber\\&&+ N_{1}(u_2) \left(a_0+\int_0^{v_2}A(t)dt 
          + \frac{1}{2k}\,  B(v_2) \right) \biggr).\nonumber
 	\label{theta}
\end{eqnarray}
The classical results on the module of syzygies on polynomial rings described in \cite{mourrain_dimension_2016} (see Proposition 4.3. in the reference), will be used in order to prove the corresponding statements in the context of syzygies on spline functions. First, we recall the notation and results concerning the polynomial case. Let $a,\, b,\, c\,$ be polynomials in $R=\RR[u]$, such that $\gcd(a,c)=\gcd(b,c)=1$, then $Z=\syz(a,b,c)$ is the  $R$-module defined by $\syz(a,b,c)=\{(A,B,C)\in \RR[u]^3\colon Aa+Bb+Cc = 0\}$. The degree of an element in $\syz(a,b,c)$ is defined as $\deg(A,B,C)=\max\{\deg(A),$ $\deg(B),\deg(C)\}$, and we are interested in studying the subspace $Z_k\subset \syz(a,b,c)$ of elements of degree less than or equal to $k-1$.
Let us denote $n=\max\{\deg(a),\deg(b),\deg(c)\}$, and
\[
	e = 
	\begin{cases}
		0\,, &\mbox{\; if 
		$ \min\bigl( n+1 - \deg(a), n - \deg(b), n - 
		\deg(c) \bigr) = 0 $\, and} \\ 
		1\,, &\mbox{\; otherwise.}\\ 
	\end{cases}
	\]
\begin{lm}\label{lm:syzygy}
Using the notation above  we have:
	\begin{itemize}
	\item $Z$ is a free $\mathbb{R}[u]$-module of rank $2$. 
	\item If $\mu$ and $\nu $ are the degree of the two free generators of $\syz(a,b,c)$ with $\mu$ minimal, then $\mu+\nu=n$.
	\item $\dim Z_k=(k-\mu+1)_++(k-n+\mu+e)_+$ where $t_+=\max(0,t)$ for any $t\in \mathbb{Z}$.
	\end{itemize}
  A basis with minimal degree corresponds  to what is called a $\mu$-basis in the literature.
\end{lm}
The proof of Lemma \ref{lm:syzygy} can  be found in \cite{mourrain_dimension_2016}.

In the following we state the analogous to Lemma \ref{lm:syzygy} in the context of syzygies on spline funtions. 
We consider $\syz_{k}^{r,r,r}$ as defined above, it is the set of spline functions $(A,B,C)\in \cU_{k-1}^{r}\times\cU_{k}^{r}\times \cU^{r}_{k}$ such that $A\,a + B\, b + C\, c =0$. An element of $\syz_{k}^{r,r,r}$  is a triple of pairs of polynomials $((A_{1},A_{2}), (B_{1},B_{2}), (C_{1},C_{2}))$.
Let $R=\RR[u]$, $R_{k}=\{p\in R\mid \deg(p)\leq k\}$, $\cQ^{r}=R/((2u-1)^{r+1})$ and $\cQ^{r}_{k}=R_{k}/((2u-1)^{r+1})$.

The elements $f=(f_{1},f_{2})$ of\, $\cU^{r+1}_{k}$ are  pairs of polynomials $f_{1},f_{2}\in R_{k}$ such that $f_{1}-f_{2}\equiv 0\  \mod (2u-1)^{r+1}$.
Let  $a=(a_1,a_2), b=(b_1,b_2), c=(c_1,c_2)\in\mathcal{U}^r$ with $\gcd(a_1,c_1)=\gcd(a_2,c_2)=\gcd(b_1,c_1)=\gcd(b_2,c_2)=1$.
We consider the following sequence:
\begin{equation}\label{diag}
  0\longrightarrow  \syz_{k}^{r,r,r}  \longrightarrow  \, \syz_{1,k}\times \syz_{2,k}\,\xrightarrow{{\;\, \phi \; }}  
  \, \cQ^{r}_{k-1}\times \cQ_{k}^{r}\times \cQ_{k}^{r}\xrightarrow{{\; \psi \;}}\cQ_{n_{1}+k}^{r}\longrightarrow 0
\end{equation}
where $\syz_{1,k}=\syz_{k}(a_{1},b_{1},c_{1})$, $\syz_{2,k}=\syz_{k}(a_{2},b_{2},c_{2})$, and 
\begin{itemize}
\item$\psi (f,g,h)=a_1f+b_1g+c_1h $,
\item$\phi (A,B,C)=(A_1-A_2,B_1-B_2,C_1-C_2) \pmod {(2u-1)^{r+1}}  $.
\end{itemize}

\begin{lm}\label{lm:exact} The sequence \eqref{diag} is exact for $k\geq n_{1}+r$ where $n_1 = \max\{\deg(a_1),$ $\deg(b_1), \deg(c_1)\}$. 
\end{lm}
\begin{proof}
Since $b_{1}, c_{1}$ are coprime, the map
  $\psi: (f,g,h)\in R_{k-1}\times R_{k}\times R_{k} \mapsto a_1f+b_1g+c_1h  \in R_{n_{1}+k}$ is surjective for $k\geq n_{1}-1$.  The map $\phi$, obtained by working modulo $(2u-1)^{r+1}$, remains surjective.

We have to prove that $\ker (\psi)=\img({\phi})$.
If $(A,B,C)\in \syz_1\times \syz_2$ then $\psi\circ\phi(A,B,C)=(A_1a_1+B_1b_1+C_1c_1)-(A_2a_1+B_2b_1+C_2c_1)=-(A_2a_1+B_2b_1+C_2c_1)$. Because $a,b,c\in \cU^{r}$, we have $a_1\equiv a_2 \pmod {(2u-1)^{r+1}}$, $b_1\equiv b_2\pmod {(2u-1)^{r+1}}$ and $c_1\equiv c_2\pmod {(2u-1)^{r+1}}$, so that
\[
\psi \circ\phi (A,B,C)\equiv -(A_2a_2+B_2b_2+C_2c_2)\equiv 0 \pmod {(2u-1)^{r+1}}.
\]
This implies that $\img(\phi) \subset \ker(\psi)$.

Conversely, if $\psi (f,g,h)=0$ with $\deg(f)\le r$, $\deg(g)\leq r$, $\deg(h)\leq r$ then $fa_1+gb_1+hc_1=d\, (2u-1)^{r+1}$ for some polynomial $d\in R$ of degree $\leq n_{1}-1$. Since $\gcd(b_1,c_1)=1$, there exists $p,q \in R_{n_{1}-1}$ such that $d=p\, b_1+q\, c_1$, we deduce that:
\[
(2u-1)^{r+1} d=(2u-1)^{r+1} \,(p\,b_1+q\,c_1)=f\,a_1+g\,b_1+h\,c_1,
\]
with $\deg((2u-1)^{r+1}p)\leq n_{1}+r$.
This yields
\begin{equation}\label{eq:syz1}
f\,a_1+(g - p (2u-1)^{r+1}) \,b_1+(h-(2u-1)^{r+1} q)\,c_1=0.
\end{equation}
Since $k\geq n_{1}+r$, this implies that $((f,0),(g-(2u-1)^{r+1}p,0),(h-(2u-1)^{r+1}q,0))\in \syz_{1,k}\times \syz_{2,k}$ and its image by $\phi$ is $(f,g,h)$. This shows that $\ker(\psi) \subset \img(\phi)$ and implies the equality of the two vector spaces.

\noindent By construction, the kernel of $\phi$ is the pair of triples $((A_{1},B_{1},C_{1}), (A_{2},B_{2},C_{2}))$ in $\syz_{1,k}\times \syz_{2,k}$ such that
$A_{1}-A_{2}\equiv B_{1}-B_{2}\equiv C_{1}-C_{2}\equiv 0 \pmod{(2u-1)^{r+1}}$, that is, the set $\syz^{r,r,r}_{k}$ of triples $(A,B,C)\in \cU_{k-1}^{r}\times \cU_{k}^{r}\times \cU_{k}^{r}$ such that $ A\,a + B\, b + C\, c =0$.
  
This show that the sequence \eqref{diag} is exact.
\end{proof}
 

We deduce the dimension formula:
\begin{prop}\label{dim}
Let $(p_1,q_1)$ (resp. $(p_2,q_2)$) be a basis of $\syz_{1}$ (resp. $\syz_{2}$) of minimal degree $(\mu_1,\nu_1)$ (resp. $(\mu_2, \nu_2)$) and $e_{1}, e_{2}$ defined as above for $(a_{1},b_{1},c_{1})$ and $(a_{2},b_{2},c_{2})$. For $k\geq \min(n_{1},n_{2})+r$,
\begin{align*}
\dim(\syz^{r,r,r}_k)=(k-\mu_1+1)_{+}&+(k-n_{1}+\mu_1+e_{1})_{+}+(k-\mu_2+1)_{+}\\
&+(k-n_{2}+\mu_2+e_{2})_{+}- \min(r+1,k) -(r+1).
\end{align*}
This dimension is denoted $d_\tau(k,r)$.  
\end{prop}
\begin{proof}
  By symmetry, we may assume that $n_{1}= \min(n_{1},n_{2})$.
  For $k\geq  n_{1}+r$, the sequence \eqref{diag} is exact and we have
  $$
  \dim \syz^{r,r,r}_{k} = \dim\syz_{1,k}+\dim \syz_{2,k} - \dim \cQ^{r}_{k-1} -2 \dim \cQ^{r}_{k} + \dim \cQ^{r}_{n_{1}+k}.
  $$
  We have  $\dim \cQ^{r}_{k-1} = \min(r+1,k)$ and $\dim \cQ^{r}_{k} = \dim \cQ^{r}_{n_{1}+k}=r+1$, since $k\geq n_{1}+r$.
  This leads to the formula, using Lemma \ref{lm:syzygy}.
\end{proof}

\subsection{Basis of the syzygy module}\label{sec:syz_basis}

The diagram \eqref{diag} allows to construct a basis for the space of syzygies $\syz_{k}^{r,r,r}$ associated to the gluing data $a,b,c\in \cU^r$. In the rest of this section we will show how to construct such a basis.

\begin{lm}
Assume that $k\geq n_{1}+r$. Using the notation of Proposition \ref{dim}, we have the following assertions:
\begin{itemize}
\item For any $\mathfrak{p}_2\in \syz_{2,k}$, there exists $\mathfrak{p}_1\in \syz_{1,k}$  such that $(\mathfrak{p}_1,\mathfrak{p}_2)\in  \ker( {\phi})$.
\item There exist $t, \; s\in \NN$ such that if $\cG=\{ (p_1(2u-1)^{i},0): 0\leq i\leq t\}\bigcup \{(q_1(2u-1)^{j},0): 0\leq j\leq l \}  $ then ${\phi} (\cG)$ is a basis of the vector space $\ker (\psi)$.
\item $\ker({\phi})\bigoplus \langle \cG\rangle=\syz_{1,k}\times \syz_{2,k}$.

\end{itemize}
\end{lm}
\begin{proof}
Let $\mathfrak{p}_2=(A_2,B_2,C_2)\in \syz_{2,k}$. 
As ${\phi}((0,\mathfrak{p}_2)) =(f, g, h)$ is in $\ker(\psi)$ (since $\psi \circ \phi=0$), we can construct $\mathfrak{p}_1 \in \syz_{1,k}$ such that 
$\phi((\mathfrak{p}_1,0))= {\phi}((0,\mathfrak{p}_2))$ as we did in the proof of Lemma \ref{lm:exact} for $(f, g, h)\in \ker(\psi)$ using relation \eqref{eq:syz1}. This gives an element of the form $(\mathfrak{p}_1,0)\in  \syz_{1,k}\times\{0\}$, and finally  $(\mathfrak{p}_1,\mathfrak{p}_2)\in  \ker( {\phi})$, this proves the first point.

The second point follows from the fact that ${\phi }(\syz _{1,k}\times \{0\}) =\ker (\psi)$ (since by Lemma \ref{lm:exact}, the sequence \eqref{diag} is exact) and that $\{(p_1(2u-1)^i,0): \;i\leq k-\mu_1\} \bigcup \{(q_1(2u-1)^j,0) j\leq k-\nu_1\}$ is a basis of $\syz_{1,k}\times \{0\}$ as a vector space, thus the image of this basis is a generating set for $\ker(\psi)$. Since it is a $R$-module, it has a basis as described in the second point of this lemma.

The third point is a direct consequence of the second one.
\end{proof}

Considering the map in (\ref{theta}), the first point of the lemma has an intuitive meaning: any function defined on a part of $\mathcal{M}_\tau $ and that satisfies  the gluing conditions imposed by $a_1,b_1,c_1$ can be extended to a function over  $\mathcal{M}_\tau$ that satisfies the gluing  conditions $a,b,c$.
The third point allows us to define the projection $\pi^r_1$ of an element on $\ker ({\phi})$ along $\langle \cG\rangle$.


%
Let  $(\tilde{p_2},p_2)$, $(\tilde{q_2}, q_2)$ be the two projections of $(0,p_2)$ and $(0, q_2)$ by $\pi^r_1$ respectively. We denote:
\begin{itemize}
\item $\cZ^{r}_1=\{(0,(2u-1)^{i}p_2): r+1 \leq i \leq k-\mu_{2}\}$
\item $\cZ^{r}_2=\{(0,(2u-1)^{i}q_2):r+1 \leq i \leq k-\nu_{2}\}$
\item $\cZ^{r}_3=\{ ((2u-1)^{i}q_1,0):r+1 \leq i \leq k-\mu_{1}\}$
\item $\cZ^{r}_4=\{ ((2u-1)^{i}p_1,0): r+1 \leq i \leq k-\nu_{2}\}$
\item $\cZ^{r}_5=\{ (2u-1)^i(\tilde{p_2},p_2):0\leq i\leq r\}$
\item $\cZ^{r}_6=\{ (2u-1)^i(\tilde{q_2},q_2):0\leq i \leq r\}$
\item $\cZ^{r}=\cZ^{r}_1 \bigcup\cZ^{r}_2 \bigcup\cZ^{r}_3\bigcup\cZ^{r}_4\bigcup\cZ^{r}_5\bigcup\cZ^{r}_6$
\end{itemize}

\begin{prop}\label{basis}  Using the notation above we have the following:
\begin{itemize}
 \item The set  $\cZ^{r}$ is a basis of the vector space $\syz^{r,r,r}_{k}$.
 
 \item The set $\cY=\{ (0,(2u-1)^{r+1}p_2),(0,(2u-1)^{r+1}q_2), (\tilde{p_2},p_2), (\tilde{q_2},q_2),((2u-1)^{r+1}q_1,0),((2u-1)^{r+1}p_1,0)\}$ is a generating set of the $R$-module $\syz^{r,r,r}$.
\end{itemize}
\end{prop}
\begin{proof}
The cardinal of $\cZ^{r}$ is equal to the dimension of $\syz^{r,r,r}_k$, we have to prove that it is a free set.
Let $\mathfrak{a}=(\mathfrak{a}_i)$, $\mathfrak{b}=(\mathfrak{b}_i)$, $\mathfrak{c}=(\mathfrak{c}_i)$, $\mathfrak{d}=(\mathfrak{d}_i)$, $\mathfrak{e}=(\mathfrak{e}_i)$, $\mathfrak{f}=(\mathfrak{f}_i)$ for $i\in \{ 0,\ldots, k\}$ a set of coefficients. Suppose that:

\begin{align*}
0 &= \sum_{i=0}^{r}\mathfrak{a}_i(2u-1)^i(\tilde{p_2},p_2)
 + \sum_{i=0}^{r}\mathfrak{b}_i(2u-1)^i (\tilde{q_2},q_2)\\
 & +\sum_{i=0}^{k-r - \nu_1}\mathfrak{c}_i ((2u-1)^{i+r+1}q_1,0) + \sum_{i=0}^{k-r-\mu_1}\mathfrak{e}_i((2u-1)^{i+r+1}p_1,0)\\
 &+
\sum_{i=0}^{k-r-\mu_2}\mathfrak{d}_i(0,(2u-1)^{r+i+1}p_2) 
 + \sum_{i=0}^{k-r-\nu_2}\mathfrak{f}_i(0,(2u-1)^{i+r+1}q_2).
\end{align*}
Then we have the following equations, 
 \begin{align}
 0  = & \sum_{i=0}^r\mathfrak{a}_i(2u-1)^i\tilde{p_2} + \sum_{i=0}^r\mathfrak{b}_i(2u-1)^i \tilde{q_2} + \sum_{i=0}^{k-r-\nu_1}\mathfrak{c}_i(2u-1)^{r+1+i}q_1\nonumber\\
 & + \sum_{i=0}^{k-r-\mu_1}\mathfrak{e}_i(2u-1)^{r+1+i}p_1 \label{eq14} \\
0 = & \sum_{i=0}^r\mathfrak{a}_i(2u-1)^ip_2+\sum_{i=0}^r\mathfrak{b}_i(2u-1)^i q_2 +
\sum_{i=0}^{k-r-\mu_2}\mathfrak{d}_i(2u-1)^{r+1+i}p_2 \nonumber\\
&+ \sum_{i=0}^{k-r-\nu_2}\mathfrak{f}_i(2u-1)^{r+1+i}q_2
\label{eq15}
\end{align}
we know that $p_2$ and $q_2$ are free generators of $\syz_2$, by (\ref{eq15}) this means that all the coefficients $\mathfrak{a}_i$, $\mathfrak{b}_i$, $\mathfrak{d}_i$, $\mathfrak{f}_i$ that are used in the equation are zero. Replacing in the equation(\ref{eq14}) we get in the same way that the other coefficients $\mathfrak{c}_i,\mathfrak{e}_i$ are zero, so the set is free. Finally since the set $\cY$ does not change when $k$ changes, then $\cY$ generates $\syz^{r,r,r}$. 
\end{proof}

We have similar results if we proceed in a symmetric way exchanging the role of the first and second  polynomial components of the spline functions. The corresponding basis of $\syz_{k}^{r,r,r}$ is denoted $\cZ'^{r}$ and the generating set of the $R$-module is 
\begin{align*}
\cY' = \bigl \{ \bigl(0,(2u-1)^{r+1}p_2\bigr),\,& \bigl(0,(2u-1)^{r+1}q_2\bigr),\, \bigl(p_1,\tilde{p_1}\bigr),\\
& \bigl(q_1,\tilde{q_1}\bigr),\, \bigl((2u-1)^{r}q_1,0\bigr),\, \bigl((2u-1)^{r}p_1,0\bigr)\bigr\}.
\end{align*}
It remains to compute the dimension and a basis for $\syz_{k}^{r-1,r,r}$, we deduce them those of $\syz_{k}^{r-1,r-1,r-1}$ and $\syz_{k}^{r,r,r}$, and it will depend on the gluing data as we explain in the following. 
\begin{prop}\label{difreg}\ 
\begin{itemize}
\item If \, $a(1/2)\neq 0$\, then \, $\syz_{k}^{r,r,r}=\syz_{k}^{r-1,r,r}$, otherwise we have that \, $\dim (\syz_{k}^{r-1,r,r})= \dim (\syz_{k}^{r,r,r})+1$.
\item For the second case, an element in $\syz_{k}^{r-1,r,r}\setminus \syz_{k}^{r,r,r}$ is of the form: $\alpha (2u-1)^r(0,p_2)+\beta (2u-1)^r(0,q_2)$, with $\alpha, \beta \in \RR$.
\end{itemize}
\end{prop}
For the proof of this proposition we need the following lemma that can be proven exactly in the same way as Proposition \ref{basis} above.
\begin{lm}\label{lem:4.7}
The set $\tilde{\cZ}^{r-1}=
{\cZ}'^{r}\bigcup \{(2u-1)^r(0,p_2),(2u-1)^r(0,q_2)\}$ is a basis of $\syz_{k} ^{r-1,r-1,r-1}$
.
\end{lm}
\begin{proof}[Proof of Proposition \ref{difreg}.]
We denote $p_1=(p_1^1,p_1^2,p_1^3)$, and $q_1=(q_1^1,q_1^2,q_1^3)$, where $p_i^j$ and $q_i^j$ are polynomials. Suppose that there exists $(A,B,C)\in \syz_{k}^{r-1,r,r}\setminus \syz_{k}^{r,r,r}$, then by the previous lemma we can choose $(A,B,C)=\alpha (2u-1)^r(0,p_2)+\beta (2u-1)^r(0,q_2)$ with $\alpha, \beta\in\RR$, that is:
$$
\left\lbrace \begin{matrix}
A=\alpha (0,(2u-1)^rp_2^1)+\beta (0,(2u-1)^rq_2^1) \\
B=\alpha (0,(2u-1)^rp_2^2)+\beta (0,(2u-1)^rq_2^2)\\
C=\alpha (0,(2u-1)^rp_2^3)+\beta (0,(2u-1)^rq_2^3)\\
\end{matrix}\right.
$$
But since  $B, C\in \cU^r$, we deduce:
$$ \left\lbrace\begin{matrix}
 (2 u-1)^{r+1}\ \mathrm{divides}\ B_2-B_1=(2 u-1)^r(\alpha p_2^2+\beta q_2^2)\\
(2 u-1)^{r+1}\ \mathrm{divides}\  C_2-C	_1=(2 u -1)^r(\alpha p_2^3+\beta q_2^3) \\
\end{matrix}\right.
$$
This means  that
$$ \left\lbrace\begin{matrix}
\alpha p_2^2(\frac{1}{2})+\beta q_2^2(\frac{1}{2})=0\\
\alpha p_2^3(\frac{1}{2})+\beta q_2^3(\frac{1}{2})=0\\
\end{matrix}\right.
$$
As the determinant of this system is exactly $p_2^2(\frac{1}{2})q_2^3(\frac{1}{2})-p_2^3(\frac{1}{2})q_2^2(\frac{1}{2})=a(\frac{1}{2})$, we deduce the two points of the proposition.
\end{proof}

Lemma \ref{lem:4.7} implies the following proposition:
\begin{prop}\label{demi}
  The dimension of $\syz_{k}^{r-1,r,r}$ is $\tilde{d}_{\tau}(k,r) = d_{\tau}(k,r)+\delta_{\tau}$ with $\delta_{\tau}=1$ if $a(\frac{1}{2})= 0$ and $0$ otherwise. 
\end{prop}
 
\subsection{Separation of vertices}\label{sec:sepvert}
We analyze now the separability of the spline functions on an edge,
that is when the Taylor map at the vertices separate the spline functions.

Let $f=(f_{1},f_{2}) \in \cR(\sigma_{1})\oplus \cR(\sigma_{2})$ of the
form $f_{i}(u_{i},v_{i})= p_{i}
+ q_{i}\, u_{i} + \tilde{q}_{i}\, v_{i} + s_{i} \, u_{i}v_{i}+ r_{i}
\,u_{i}^{2} + \tilde{r}_{i} \, v_{i}^{2}+ \cdots$. Then
\[ 
	T_{\gamma}(f)=[p_{1}, q_{1},\tilde{q}_{1},s_{1},p_{2},
	q_{2},\tilde{q}_{2},s_{2}].
\]
If $f=(f_{1},f_{2})\in \cS^{r}_{k}(\cM_{\tau})$, then 
taking the Taylor expansion of the gluing condition 
\eqref{eq:edgecond}  centered at $u_{1}=0$ yields
\begin{eqnarray}
	{q}_{2} + s_{1}\, u_{1} &=& (\ma(0) +\ma'(0) u_{1}+ \cdots)
	\, ( \tilde{q}_{2} + 2\, \tilde{r}_{2}\, u_{1} + \cdots) \label{eq:exptaylor}\\
	&&+ (\mb(0) +\mb'(0) u_{1}+ \cdots) 
	\, ( q_{2} + s_{2}\, u_{1} +\cdots )\nonumber
\end{eqnarray}
Combining \eqref{eq:exptaylor} with \eqref{eq:edgecond0} yields 
\begin{eqnarray*}
	p_{1} & = & p_{2} \\ 
 	q_{1} & = & \tilde{q}_{2} \\
 	r_{1} & = & \tilde{r}_{2} \\
	\tilde{q}_{1} & = & \ma(0) \, \tilde{q}_{2} + \mb(0)\, q_{2}\\
	s_{1} &=& 2\,\ma (0)\, \tilde{r}_{2} + \mb (0) \, s_{2}
 	+ \ma'(0) \, \tilde{q}_{2}+ \mb'(0) \, q_{2}. 
\end{eqnarray*}
Let $\cH(\gamma)$ be the linear space spanned
by the vectors $[p_{1}, q_{1},\tilde{q}_{1},s_{1},p_{2}, q_{2},\tilde{q}_{2},s_{2}]$,
which are solution of these equations.

If $\ma(0)\neq 0$, it is a space of dimension $5$ otherwise its
dimension is $4$. Thus $\dim \cH(\gamma)=5 -\cross_{\tau}(\gamma)$.
%


 In the next proposition we use  the  notation of the previous section. 
 \begin{prop}
For $k\geq \nu_1+1$  we have $T_\gamma (\cS^{1,r}_{k}(\mathcal{M}_\tau))=\mathcal{H}(\gamma)$. In particular $\dim (T_\gamma  (\cS^{1,r}_{k}(\mathcal{M}_\tau)))=5-c_\tau(\gamma)$.
\end{prop}

\begin{proof}
 By construction we have $T_\gamma (\cS^{1,r}_{k}(\mathcal{M}_\tau))\subset \mathcal{H}(\gamma)$. Let us prove that they have the same dimension.
 If $(A,B,C)\in \syz^{r,r,r}_{k}$ with $A=(A_1,A_2)$,$B=(B_1,B_2)$,$C=(C_1,C_2)$, then $(A_1,B_1,C_1)$ is an element of the $R$-module spanned by $p_1=(p_1^1,p_1^2,p_1^3)$, $q_1=(q_1^1,q_1^2,q_1^3)$, ie $(A,B,C)=a_{1} ((1-2u)^{r+1}p_{1},0) +P(p_1,\tilde{p_1})+Q(q_1,\tilde{q_1})$.
 Let $f=(f_1,f_2)=\Theta _\tau(a_{0},(A,B,C))$ (see \eqref{eq:deftheta}), then it is easy to see that:
\begin{align}
	& T_\gamma (f) = \left[
	\begin{array}{c}
 	f_{1}(\gamma)\\
 	\partial_{u_{1}}f_{1}(\gamma)\\ 
 	\partial_{u_{2}}f_{2}(\gamma)\\ 
 	-\partial_{v_{1}}f_{1}(\gamma)\\
	\partial_{u_{2}} \partial_{v_{2}}f_{2}(\gamma)\\ 
	-\partial_{u_{1}} \partial_{v_{1}}f_{1}(\gamma)\\ 
	\end{array}
	\right] \label{matr} \\
	&=
 	\left[
	\begin{array}{cccccc}
 		1 & 0 & 0 &0 &0&0\\
 		0 & p^1_{1}(0) & p^1_{1}(0) & q^1_{1}(0)& 0 &0\\
 		0 & p^2_{1}(0) & p^2_{1}(0) & q^2_{1}(0)& 0 &0\\
 		0 & p^3_{1}(0) & p^3_{1}(0) & q^3_{1}(0)& 0 &0\\
 		0 & {p^2_{1}}'(0)  - 2(r+1) p^2_{1}(0)& {p^2_{1}}'(0) & {q^2_{1}}'(0) & p^2_{1}(0) & q^2_{1}(0)\\
 		0 & {p^3_{1}}'(0)  - 2(r+1) p^3_{1}(0) & {p^3_{1}}'(0) & {q^3_{1}}'(0) & p^3_{1}(0) & q^3_{1}(0)\\
	\end{array}
	\right]
	\left[
	\begin{array}{c}
 		a_{0}\\ 
 		a_{1}\\ 
 		P(0)\\ 
 		Q(0)\\ 
 		P'(0)\\ 
 		Q'(0)\\ 
	\end{array} 
	\right]	\nonumber
      \end{align}
The second column of the matrix is linearly dependent on the third and fifth columns. 
Using the same argument as in the proof of \cite[Proposition 4.7]{mourrain_dimension_2016} on the first and 4 last columns of this matrix, we prove that its rank is $5-c_\tau ^\gamma$. 
By taking $P,Q\in R_{1}$ of degree $\le 1$, which implies that $k\geq \max(\deg(P \, p_{1}), \deg(Q\, q_{1}))=\nu_{1}+1$, the vector $\left[
  a_{0}, 
 		P(0), 
 		Q(0), 
 		P'(0), 
 		Q'(0)\right]$ can take all the values of $\mathbb{R}^5$ and we have $T_\gamma (\cS^{1,r}_{k}(\mathcal{M}_\tau))=\mathcal{H}(\gamma)$. This ends the proof.
\end{proof}





We consider now the separability of the Taylor map at the two end points $\gamma,\gamma'$.
\begin{prop}\label{prop:Hgg}
  Assume that $k\geq \max(\nu_1+2,\nu_2+2, \mu_{1}+r+1, \mu_{2}+r+1)$. Then $T_{\gamma,\gamma'}(\cS^{1,r}_{k}(\mathcal{M}_\tau))= (\mathcal{H}(\gamma),\mathcal{H}(\gamma'))$
  and $\dim T_{\gamma,\gamma'}(\cS^{1,r}_{k}(\mathcal{M}_\tau)) = 10 - \cross_{\tau}(\gamma) -\cross_{\tau}(\gamma')$.
\end{prop}

\begin{proof}
The inclusion $T_{\gamma,\gamma'}(\cS^{1,r}_{k}(\mathcal{M}_\tau))\subseteq (\mathcal{H}(\gamma),\mathcal{H}(\gamma'))$ is clear by construction.
For the converse,  we show that the image of $T_{\gamma,\gamma'} \circ \Theta_{\tau}$ contains $(\mathcal{H}(\gamma),0)$ and then by symmetry we have that $(0,\mathcal{H}(\gamma ))$ is in the image of $T_{\gamma,\gamma'}\circ \Theta_{\tau}$. Let $f=(f_1,f_2)=\Theta_\tau(a_{0},(A,B,C))\in \mathcal{S}^{1,r}_k(\mathcal{M}_\tau)$ with
  $(A,B,C)=a_{1} ((1-2u)^{r+1}p_{1},0) +P(p_{1},\tilde{p_1})+Q(q_{1},\tilde{q}_1)$ and $P,Q\in \cU_{2}^{r}$.
  The image of $f$ by $T_{\gamma}$ is of the form  \eqref{matr}. The image of $f$ by $T_{\gamma'}$ is of the form
\begin{align*}
& T_{\gamma'} (f) = \left[
	\begin{array}{c}
 	f_{1}(\gamma')\\
 	\partial_{u_{1}}f_{1}(\gamma')\\ 
 	\partial_{u_{2}}f_{2}(\gamma')\\ 
 	-\partial_{v_{1}}f_{1}(\gamma')\\
	\partial_{u_{2}} \partial_{v_{2}}f_{2}(\gamma')\\ 
	-\partial_{u_{1}} \partial_{v_{1}}f_{1}(\gamma')\\ 
	\end{array}
	\right]\\
	&=
 	\left[
	\begin{array}{cccccc}
	1 & t_{1} & 0  & 0   &0&0\\
 	0 & 0 & \tilde{p}^1_{1}(1) & \tilde{q}^1_{1}(1)& 0 &0\\
 	0 & 0 & \tilde{p}^2_{1}(1) & \tilde{q}^2_{1}(1)& 0 &0\\
 	0 & 0 & \tilde{p}^3_{1}(1) & \tilde{q}^3_{1}(1)& 0 &0\\
 	0 & 0 & {\tilde{p}^{2}_{1}}\!\,'(1) & {\tilde{q}^{2}_{1}}\!\,'(1) & \tilde{p}^2_{1}(1) & \tilde{q}^2_{1}(1)\\
 	0 & 0 & {\tilde{p}^3_{1}}\!\,'(1) & {\tilde{q}^3_{1}}\!\,'(1) & \tilde{p}^3_{1}(1) & \tilde{q}^3_{1}(1)\\
	\end{array}
	\right]
	\left[
          \begin{array}{c}
 		a_{0}\\ 
                a_{1} \\ 
 		P(1)\\ 
 		Q(1)\\ 
 		P'(1)\\ 
          Q'(1)\\
	\end{array} 
      \right]
      + \left[
	\begin{array}{c}
          L_{1}(P) + L_{2}(Q)\\
          0 \\
          0 \\
          0 \\
          0 \\
          0 \\
	\end{array}
	\right]
\end{align*}
      with $t_{1}=\int_{0}^{1/2} (1-2u)^{r+1}p_{1}^{1} du$, $L_{1}(P)=\int_{0}^{1} P\,\tilde{p}_{1}^{1} du$,  $L_{2}(Q)=\int_{0}^{1} Q\,\tilde{q}_{1}^{1} du$.
      By choosing $P(1)=P'(1)=Q(1)=Q'(1)=0$ and $a_{0}+t_{1}a_{1}=0$, we have an element in the kernel of this matrix. By choosing $a_{0}, P(0), P'(0), Q(0), Q'(0)$ and $a_{1}$ such that $a_{0}+t_{1}a_{1}+L_{1}(P)+ L_{2}(Q)=0$, we can find a solution to the system \eqref{matr} for any $f\in \cS_{k}(\cM_{\tau})$.
      Therefore, constructing spline coefficients $P,Q\in \cU_{2}^{r}$ which interpolate  prescribed values and derivatives at $0,1$, we can construct spline functions $f\in \cS_{k}(\cM_{\tau})$ such that $T_{\gamma}(f)$ span $\cH(\gamma)$ and $T_{\gamma'}(f)=0$. The degree of the spline is $k\geq \max(\nu_{1}+2, \mu_{1}+r+1)$.
By symmetry, for $k\geq \max(\nu_{2}+2, \mu_{2}+r+1)$, we have  $(0,\cH(\gamma'))\subset T_{\gamma,\gamma'}(\cS_{k}^{1}(\cM_{\tau})$, which concludes the proof.
\end{proof}

\begin{df} The separability $\ms(\tau)$ of the edge $\tau$ is the minimal
  $k$ such that $T_{\gamma,\gamma'}(\cS^{1,r}_{k}(\cM_{\tau}))=
	(T_{\gamma}(\cS^{1,r}_{k}(\mathcal{M}_{\tau})), T_{\gamma'}(\cS^{1,r}_{k}(\mathcal{M}_{\tau})))$.
      \end{df}
 The previous proposition shows that $\ms(\tau)\leq \max(\nu_1+2,\nu_2+2, \mu_{1}+r+1, \mu_{2}+r+1)$.
\subsection{Decompositions and dimension} \label{sec:polybasis}

Let $\tau\in\cM_{1}$ be an interior edge $\tau$ shared by the cells
$\sigma_{0}, \sigma_{1}\in \cM_{2}$. 
The Taylor map along the edge $\tau$ of $\cM_{\tau}$ is
\begin{align*}
D_{\tau}\colon \cR_{k}(\sigma_{0})\oplus \cR_{k}(\sigma_1)&\rightarrow\cR_{k}(\sigma_{0})\oplus \cR_{k}(\sigma_1)\\
(f_{0}, f_{1})&\mapsto (D_{\tau}^{\sigma_0} (f_0), D_{\tau}^{\sigma_1} (f_1)\bigr).
\end{align*}
Its image is the set of splines of  $\cR_{k}^r({\sigma_{1}})\oplus \cR_{k}^r({\sigma_{2}})$
with {\em support along} $\tau$.
The kernel is the set of splines of  $\cR_{k}^r({\sigma_{1}})\oplus \cR_{k}^r({\sigma_{2}})$
with vanishing b-spline coefficients {along} the edge $\tau$. The elements of $\ker(D_{\tau})$ are smooth splines in $\cS^{r}_{k}(\cM_{\tau})$.
Let $W_{k}(\tau)=D_{\tau}(\cS^{r}_{k}(\cM_{\tau}))$. It is the set of splines in $\cS^{r}_{k}(\cM_{\tau})$ with a support along $\tau$.  
As $D_{\tau}$ is a projector, we have the decomposition
\begin{equation} \label{eq:S}
\cS^{r}_{k}(\cM_{\tau}) = \ker(D_{\tau}) \oplus W_{k}(\tau).
\end{equation}
From the relations \eqref{eq18} and \eqref{eq19}, we deduce that $W_{k}(\tau)=\im \Theta_{\tau}$.
%
%
%
Since $\Theta_{\tau}$ is injective, thus $\dim(W_{k}(\tau))=\dim \syz^{r,r,r}_{k-1} +1 =d_{\tau}(k,r)+1$ and 
$W_{k}(\tau) \neq \{0\}$ when $k\geq \mu_1 $ and $k\geq \mu_2 $
(Lemma (\ref{lm:syzygy}) (iii)). 


The map $T_{\gamma,\gamma'}$ defined in Section~\ref{sec:tay} induces the
exact sequence
\begin{equation} \label{eq:Kexact}
	0 \rightarrow \cK_{k}(\tau) \rightarrow \cS_{k}^{1,r}(\cM_{\tau}) 
	\stackrel{T_{\gamma,\gamma'}}{\longrightarrow} \cH(\tau) \rightarrow 0
\end{equation}
where $\cK_{k}(\tau)= \ker(T_{\gamma,\gamma'})$
and $\cH(\tau)=T_{\gamma,\gamma'}(\cS_{k}^{1,r}(\cM_{\tau}))$.
\begin{df}\label{def:edgespline}
	For an interior edge $\tau\in \cM_{1}^{o}$, let 
	$\cE_{k}(\tau)=\ker (T_{\gamma,\gamma'})\cap W_{k}(\tau)= \ker (T_{\gamma,\gamma'})\cap \im D_{\tau}$ be the
	set  of splines in $\cS_{k}^{r}(\cM_{\tau})$ with their 
	support along $\tau$ and with vanishing 
	Taylor expansions at $\gamma$ and $\gamma'$.
	For a boundary edge $\tau'=(\gamma,\gamma')$, which belongs to a face $\sigma$, we 
	also define $\cE_{k}(\tau')$ as the set of elements of $\cR_{k}^r(\sigma)$
	with their  support along $\tau'$ and with vanishing 
	Taylor expansions at $\gamma$ and $\gamma'$.
\end{df}
Notice that the elements of $\cE_{k}(\tau)$ have their support along
$\tau$ and that their Taylor expansion at $\gamma$ and $\gamma'$
vanish. Therefore, their Taylor expansion along all (boundary) edges of
$\cM_{\tau}$ distinct from $\tau$ also vanish.

As $\ker(D_{\tau})\subset \cK_{k}(\tau)$, we have the decomposition
\begin{equation} \label{eq:K}
\cK_{k}({\tau}) = \ker(D_{\tau})\oplus \cE_{k}(\tau).
\end{equation} 
We deduce the following result 
\begin{lm}\label{lm:dim:E}For an interior edge $\tau\in \cM_{1}^{o}$ and for $k\ge \ms(\tau)$, 	the dimension of $\cE_{k}(\tau)$ is 	\[ 		\dim \cE_{k}(\tau) = \tilde{d}_{\tau}(k,r) -9 + \cross_{\tau}(\gamma) +\cross_{\tau}(\gamma'). 	\]
\end{lm}
\begin{proof}
From the relations \eqref{eq:S},  \eqref{eq:Kexact} and \eqref{eq:K}, we have
  \begin{eqnarray*}
    \dim \cE_{k}(\tau) & = &\dim \cK_{k}(\tau) - \dim  \ker(D_{\tau})\\
      & =& \dim \cS_{k}^{1,r}(\cM_{\tau})  - \dim \cH_{k}(\tau) - \dim \cS_{k}^{1,r}(\cM_{\tau}) +\dim W_{k}(\tau)\\
      & =& \dim W_{k}(\tau) - \dim \cH_{k}(\tau),
  \end{eqnarray*}
which gives the formula using Proposition \ref{prop:Hgg}.
\end{proof}
\begin{remark}\label{rem:boundary:edge}
	When $\tau$ is a boundary edge, which belongs to the face $\sigma\in
	\cM_{2}$, we have $\cS_{k}^r(\cM_{\tau})=\cR_{k}^r(\sigma)$
        and $\dim \cE_{k}(\tau) = 2(m+1) -8= 4 k- 2r -6$.

\end{remark}
 
\subsection{Basis functions associated to an edge}\label{sec:basis_edge}
Suppose that $\cB_k^r=\{\beta_i\}_{i=0...l}$ with $l=\dim \syz ^{r-1,r,r}_{k-1} $ and $\beta_i=(\beta_i^1,\beta_i^2,\beta_i^3)$, is a basis of $\syz ^{r-1,r,r}_{k-1} $. We know also that $\cE _k =\{f=\Theta_{\tau}(a_0,(A,B,C)): T_{\gamma,\gamma'}(f)=0, (A,B,C) \in \syz ^{r-1,r,r}_{k-1}\}$, but we have:
\begin{align*}
&T_{\gamma ,\gamma'}(f)=\begin{pmatrix}
T_{\gamma }\\
T_{\gamma '}\\
\end{pmatrix}\\
& = \begin{pmatrix}
\hspace{2.2cm} c_{0} ,A(0),-C(0),-C'(0),\hspace{2.3cm} c_{0},B(0),A(0),B'(0)\\
c_{0} + \int_{0}^{1}A(u) du , A(1),-C(1),\hspace{0.4cm} C'(1),c_{0}+\int_{0}^{1}A(u) du,B(1),A(1),B'(1)\\
\end{pmatrix}
\end{align*}
Suppose that $(A,B,C)=\bigl(\sum b_i \beta_i^1, \sum b_i \beta_i^2,\sum b_i \beta_i^3\bigr)$ with $ b_i \in \RR$, then $ T_{\gamma,\gamma'}(f)=0$ is equivalent to the system:
\begin{equation}\label{eq:B1}
\left\lbrace \begin{matrix}
a_0=0\\
\sum b_i \beta^1_i(0)=0\\
\sum b_i \beta^2_i(0)=0\\
\sum b_i \beta^3_i(0)=0\\
\sum b_i \beta^{2'}_i(0)=0\\
\sum b_i \beta^{3'}_i(0)=0\\
\sum b_i \int_0^1\beta_i(t)dt=-a_0\\
\end{matrix}      
\right.
\left\lbrace 
\begin{matrix}
\sum b_i \beta^1_i(1)=0\\
\sum b_i \beta^2_i(1)=0\\
\sum b_i \beta^3_i(1)=0\\
\sum b_i \beta^{2'}_i(1)=0\\
\sum b_i \beta^{3'}_i(1)=0\\
\end{matrix}      \right.
\end{equation}
The system \eqref{eq:B1} directly depends on the gluing data  \eqref{eq:transmap} along the edge via equations \eqref{eq18} and \eqref{eq19}, see Section \ref{sec:relation_syz} above. An explicit solution requires the computation of a basis for the syzygy module, which is constructed in Section \ref{sec:syz_basis}. 
The image by $\Theta_{\tau}$ (defined in \eqref{eq:deftheta})
of a basis of the solutions of this system yields a basis of $ \cE_k$. 
\section{Splines around a vertex}
In this section,  we analyse the spline functions, attached to a vertex, that is, the spline functions which Taylor expansions along the edges around the vertex vanish. We analyse the image of this space by the Taylor map at the vertex, and construct a set of linearly independant spline functions, which images span the image of the Taylor map. These form the set of basis functions, attached to the vertex.

Let us consider a topological surface $\cM_\gamma$ composed by quadrilateral faces $\sigma_1,\dots,\sigma_{F(\gamma)}$ sharing a single vertex $\gamma$, and such that the faces $\sigma_{i}$ and $\sigma_{i-1}$ have a common edge $\tau_{i}=(\gamma,\delta_{i})$, for $i=2,\dots,F(\gamma)$. 
If $\gamma$ is an interior vertex then we identify the indices modulo $F(\gamma)$ and $\tau_1$ is the common edge of $\sigma_{F(\gamma)}$ and $\sigma_1$, see Fig.~\ref{fig:5quad}.

\begin{figure}[ht]
\begin{center}
\includegraphics[width=4.5cm]{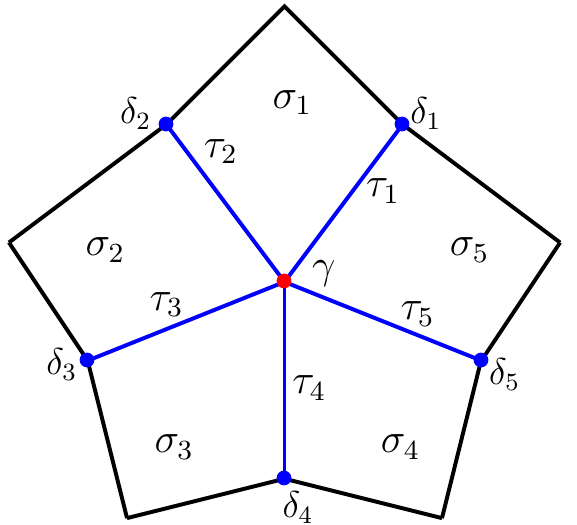}
\end{center} 
\vspace{-0.2cm}
\caption{Topological surface $\cM_\gamma$ composed by $F(\gamma)=5$ quadrilateral faces glued around the vertex $\gamma$.}\label{fig:5quad}
\end{figure}

The gluing data attached to each of the edges $\tau_i$ will be denoted by $\mathfrak{a_i}=\frac{a_i}{c_i}$, $\mathfrak{b_i}=\frac{b_i}{c_i}$. 
By a change of coordinates we may assume that $\gamma$ is at the origin $(0,0)$, and the edge $\tau_i$ is on the line $v_{i}=0$, where $(u_{i-1},v_{i-1})$ and $(u_i,v_i)$ are the coordinate systems associated to $\sigma_{i-1}$ and $\sigma_{i}$, respectively. 
Then the transition map at $\gamma$ across $\tau_i$ from $\sigma_i$ to $\sigma_{i-1}$ is as given by
\[
\phi_{\tau_i}\colon (u_i,v_i)\rightarrow 
\begin{pmatrix}
v_i\mathfrak{b_i}(u_i)\\
u_i + v_i \mathfrak{a}_i(u_i)
\end{pmatrix};
\]
following the notation in \eqref{eq:transmap}, we have $\phi_{\tau_i}=\phi_{i-1,i}$.


The restriction along the boundary edges of $\cM_\gamma$ is defined by
\begin{align*}
D_{\gamma}\colon \bigoplus_{i=1}^{F(\gamma)} \cR(\sigma_i)&\rightarrow \bigoplus_{\substack{{\tau\in\partial\cM_\gamma}\\ \tau\not\ni \gamma}}\cR^{\sigma_i}(\tau)\\
(f_i)_{i=1}^{F(\gamma)}&\mapsto \bigl(D_{\tau}^{\sigma_i} (f_i)\bigr)_{\tau\not\ni\gamma}
\end{align*}
where $D_\tau^{\sigma_i}$ is the Taylor expansion along $\tau$ on $\sigma_i$, see Section \ref{sec:tay}. 

Let $\cV_k(\gamma)$ be the set of spline functions of degree $\leq k$ on $\cM_\gamma$ that vanish at the first order derivatives along the boundary edges:
\begin{equation}\label{eq:Vgamma}
\cV_k(\gamma) =\ker D_{\gamma}\cap \cS_k^{1}(\cM_\gamma).
\end{equation}
The gluing data and the differentiability conditions in \eqref{eq:edgecond} lead to conditions on the coefficients of the Taylor expansion of  $f_i$, namely
\begin{equation}\label{eq:Texp}
f_i(u_i,v_i) = p + q_i u_i + q_{i+1} v_i + s_i u_iv_i + r_i u_i^2 + r_{i+1} v_i^2+\cdots
\end{equation}
with $p,q_i,s_i,r_i\in\R$, and for $i=2,\dots,F$ the following two conditions are satisfied 
\begin{align}
q_{i+1} & = \mathfrak{a}_i(0)q_i + \mathfrak{b}_i(0) q_{i-1}\label{eq:conditions_gamma1}\\
s_i & = 2\mathfrak{a}_i(0)r_i + \mathfrak{b}_i(0)s_{i-1} + \mathfrak{a}'_{i}(0)q_i + \mathfrak{b}'_i(0) q_{i-1}. \label{eq:conditions_gamma2}
\end{align}
Let $\mathcal{H}(\gamma)$ be the space spanned by the vectors $\mathbf{h}= [p,q_1,\dots,q_{F(\gamma)},s_1,\dots,s_{F(\gamma)}]$ such that $p$, $q_1,\dots,q_{F(\gamma)}$, $s_1,\dots,s_{F(\gamma)}$, $r_1,\dots,r_{F(\gamma)}\in \R$ give a solution for \eqref{eq:conditions_gamma1} and \eqref{eq:conditions_gamma2}.
The following result was proved in \cite[Proposition 5.1]{mourrain_dimension_2016} in the case of polynomial splines. 
\begin{prop}\label{prop:dim_H}
For a topological surface $\cM_\gamma$ consisting of $F(\gamma)$ quadrangles glued around an interior vertex $\gamma$, 
\[\dim \cH(\gamma) = 3+F(\gamma)-\sum_{\tau\ni \gamma}\mathfrak{c}_\tau(\gamma) + \mathfrak{c}_+(\gamma),\] 
where $\mathfrak{c}_\tau(\gamma)$, $\mathfrak{c}_+(\gamma)$ are as in Definition \ref{def:crossing}. 
\end{prop}
Since the vectors in $\cH(\gamma)$ only depend on the Taylor expansion of $f$ at $\gamma$, and $f$ can be seen as a polynomial spline in a neighborhood of $\gamma$, then the proof of Proposition \ref{prop:dim_H} follows the same argument as the one in \cite{mourrain_dimension_2016}.

\begin{prop}\label{prop:THgamma}
For a topological surface $\cM_\gamma$ as before, if $\ms(\tau_i)$ denotes the separability of the edge $\tau_i$ as in Definition 4, then 
\[T_\gamma\bigl(\cV_k(\gamma)\bigr)=\cH(\gamma),\]
for every $k\geq \max \{\ms(\tau_i)\colon i=1,\dots,F(\gamma)\}$.
\end{prop}
\begin{proof}
By definition (see \eqref{eq:Vgamma}), the elements of $\cV_k(\gamma)$ satisfy the conditions \eqref{eq:conditions_gamma1} and \eqref{eq:conditions_gamma2} on the Taylor expansion of $f$, then $T_\gamma\bigl(\cV_k(\gamma)\bigr)\subseteq \cH(\gamma)$.

Let us consider a vector $\bh=[p,q_1,\dots,q_{F(\gamma)},s_1,\dots,s_{F(\gamma)}]\in\cH(\gamma)$, we need to prove that this vector is in the image $T_\gamma\bigl(\cV_k(\gamma)\bigr)$.
In fact, by Proposition \ref{prop:Hgg} applied to $\tau_i=[\gamma,\delta_{i}]$,  there exists $(f^{\tau_i}_i,f^{\tau_i}_{i-1})\in \cS_k^{1,r}(\cM_{\tau_i})$ such that $T_{\gamma}(f^{\tau_i}_i,f^{\tau_i}_{i-1})=[p,q_i,q_{i+1},s_i,p,q_{i-1},q_i,s_{i-1}]$ and   $T_{\delta_i}(f^{\tau_i}_i,f^{\tau_i}_{i-1})=0$ for $k\geq\ms(\tau_i)$, for $i=2,\dots, F$. 
Let us notice that in such case, $T_\gamma^{\sigma_i}(f^{\tau_i}_i)=T_\gamma^{\sigma_i}(f^{\tau_{i+1}}_i)$. Thus, it follows that there exists $g_i\in\cR_k(\sigma_i)$ such that $T_{\tau_i}^{\sigma_i}(g_i)=f^{\tau_i}_i$ and $T_{\tau_{i+1}}^{\sigma_i}(g_i)=f^{\tau_{i+1}}_i$. The spline $g_i$ is constructed by taking the coefficients of $f^{\tau_i}_i$ and $f^{\tau_{i+1}}_i$ in $\cR^{\sigma_i}(\tau_i)$ and $\cR^{\sigma_{i}}(\tau_{i+1})$, respectively (see Section \ref{sec:tay}). 
Since $T_{\delta_i}^{\sigma_i}(f^{\tau_i}_i)=T_{\delta_i}^{\sigma_i}(g_i)=0$ and $T_{\delta_{i+1}}^{\sigma_i}(f^{\tau_{i+1}}_i)=T_{\delta_{i+1}}^{\sigma_i}(g_i)=0$ then $T_\tau^{\sigma_i}(g_i)=0$ for every edge $\tau\in\sigma_i$ such that $\gamma\notin\tau$.
Let $\bg=[g_1,g_2,\dots,g_{F(\gamma)}]$ where $g_i\in\cR_k(\sigma_i)$ is as previously constructed. Then $\bg$ and their first derivatives vanish on the edges in $\partial \cM_\gamma$, and $\bg$ satisfies the gluing conditions along all the interior edges $\tau_i$ of $\cM_\gamma$, i.e. $\bg\in\cS_k^1(\cM_\gamma)\cap \ker D_{\gamma}$. Hence $\bg\in\cV_k(\gamma)$, and by construction $T_\gamma(\bg)=\bh$. 
\end{proof}

Given a topological surface $\cM$, let $T$ be the Taylor map at all the vertices of $\cM$, as defined in Section \ref{sec:tay}. We have the following exact sequence 
\begin{equation}\label{eq:T}
0\rightarrow \cK_k(\cM)\rightarrow \cS_k^{1}(\cM)\xrightarrow{ T }\cH_{k}(\cM) \rightarrow 0
\end{equation}
where $\cH_{k}(\cM)=T\bigl(\cS_k^{1}(\cM)\bigr)$ and $\cK_k(\cM)=\ker T \cap \cS_k^{1}(\cM)$.
Let us define $\ms^* =\max \{\ms(\tau)\colon \tau\in \cM_1\}$. From Proposition \ref{prop:Hgg}, we know that $\ms^*\leq 2+ \max \{v_i^\tau\colon \text{ for } i=1,2 \text{ and } \tau\in\cM_1 \} +\min (3,r)$, where $(u_i^\tau, v_i^\tau)$ for $i=1,2$ are the degrees of the generators of $\syz_1$ and $\syz_2$, respectively, with $u_i^\tau\leq v_i^\tau$.

\begin{prop}\label{prop:dim_gamma}
Let $F(\gamma)$ and $\cH(\gamma)$ be as defined above for each vertex $\gamma\in\cM_0$, then for every $k\geq \ms^*$ we have 
$T(\cS_k^{1}(\cM))=\prod_\gamma \cH(\gamma)$
and
\[\dim T (\cS_k^{1}(\cM))=\sum_{\gamma\in \cM_0} (F(\gamma)+3) - \sum_{\gamma\in\cM_0}\sum_{\tau\ni\gamma}\cross_\tau(\gamma)+\sum_{\gamma\in\cM_0} \cross_+(\gamma).\]
\end{prop}
\begin{proof}
The statement follows directly applying Propositions \ref{prop:THgamma} and \ref{prop:dim_H} to each vertex $\gamma\in\cM_0$, with $M_\gamma$ the sub-mesh of $\cM$ which consists of the quadrangles in $\cM$ containing the vertex $\gamma$.
\end{proof}
\subsection{Basis functions associated to a vertex}
Given a topological surface $\cM$, for each vertex $\gamma\in\cM_0$, let us consider the sub-mesh $\cM_\gamma$ consisting of all the faces $\sigma\in\cM$ such that $\gamma\in\sigma$, as before, we denote this number of such faces by $F(\gamma)$.
From Proposition \ref{prop:dim_gamma} we know the dimension of $T(\cS_k^{1}(\cM))$ for $k\geq \ms^*$. In the following, we construct a set of linearly independent splines $\cB_0\subseteq\cS_k^1(\cM)$ such that $\spanset\{T(f)\colon f\in \cB_0\}=T(\cS_k^1(\cM))$.   

Let us take a vertex $\gamma\in\cM_0$ and consider the b-spline representation of the elements $f_\sigma\in \cR_k(\sigma)$ for $\sigma\in\cM_\gamma$. We construct a set $\cB_{0}(\gamma)\subset \cS_k^1(\cM_\gamma)$  of linearly independent spline function as follows:
\begin{itemize}
\item First we add one basis function $f$ attached to the value at 
  $\gamma$, such that $T_\gamma^\sigma(f_\sigma)(\gamma)=1$ for every $\sigma\in\cM_\gamma$. Let us notice that if we define $g_\sigma=\sum_{0\leq i,j\leq 1} N_i(u_\sigma)N_j(v_\sigma)$ for every $\sigma\in\cM_\gamma$, and $g$ on $\cM_\gamma$ such that $g|_\sigma=g_\sigma$, then $g(\gamma)=1$. We lift $g$ to a spline $f$ on $\cM_\gamma$ such that $f$ is in the image of the map $\Theta_\tau$ defined in \eqref{eq:deftheta}, for every $\tau\in\cM_1$ attached to $\gamma$.
%
\item We add two basis functions $g,h$ supported on $\cM_\gamma$ and attached to the first derivatives at $\gamma$. Namely, let us consider  
$ g_{\sigma_1} = (1/2k)\bigl(N_0(u_{\sigma_1}) + N_1(u_{\sigma_1})\bigr) N_1(v_{\sigma_1}) $, 
and 
$ h_{\sigma_1} = (1/2k) N_1(u_{\sigma_1})\bigl(N_0(v_{\sigma_1}) + N_1(v_{\sigma_1})\bigr) $. 
The conditions \eqref{eq:conditions_gamma1} and \eqref{eq:conditions_gamma2} allow us to find~$g_{\sigma_i}$ and $h_{\sigma_i}$, for $i=2,\dots, F(\gamma)$ from $g_{\sigma_1}$ and $h_{\sigma_1}$, respectively. Thus, we define $g$ and $h$ on $\cM_\gamma$ by taking $g|_\sigma=g_\sigma$  and $h|_\sigma=h_\sigma$. Since $g$ and $h$ by construction satisfy the gluing conditions \eqref{eq:edgecond0} and \eqref{eq:edgecond} along the edges, then they are splines in the image $\cS_k^1(\cM_\gamma)$ of $\Theta_{\tau}$ for every interior edge $\tau\in\cM_\gamma$.  
%
\item For each edge $\tau_i$ for $i=1,\dots,F(\gamma)$, let us define the function $g_{\sigma_i} = c_{1,1}^{\sigma_i}(g_{\sigma_i}) N_{1}(u_{\sigma_i})N_{1}(v_{\sigma_i})$, where 
$c_{1,1}^{\sigma_i}(g_{\sigma_i})=1/4k^2$ if $\tau_i$ is not a crossing edge, and equal to zero otherwise.
Then, for every fix edge $\tau_i\in\cM_\gamma$ attached to $\gamma$ we construct a spline
$g$ on $\cM_\gamma$ such that $g|_{\sigma_i}=g_{\sigma_i}$, and $g|_{\sigma_j}$ for $j\neq i$ are determined by $g_{\sigma_i}$ and the gluing data at $\gamma$, according to \eqref{eq:conditions_gamma1} and \eqref{eq:conditions_gamma2}. 
The previous construction produces $F(\gamma)-\sum_{\tau\ni\gamma}\cross_\tau(\gamma)$ (non-zero) spline functions. These splines,  by construction, are in the image of $\Theta_\tau$ \eqref{eq:deftheta} along all the edges $\tau\in\cM_1$ attached to $\gamma$.
\item If $\gamma$ is a crossing vertex, by definition all the edges attached to $\gamma$ are crossing edges. In this case, we define $g_{\sigma_1} = (1/4k^2)N_{1}(u_{\sigma_1})N_{1}(v_{\sigma_1})$, and determine $g_{\sigma_i}$ for $i=2,\dots, F(\gamma)$ using the gluing data at $\gamma$ and conditions \eqref{eq:conditions_gamma1} and \eqref{eq:conditions_gamma2}. Defining $g$ on $\cM_\gamma$ by $g|_{\sigma_i}=g_{\sigma_i}$ we obtain a spline in $\cS_k^1(\cM_\gamma)$.  
\end{itemize}
Let us notice that if $\tau_i$ is a crossing edge then, following the notation in the Taylor expansion of $g_i(u_i,v_i)$ in \eqref{eq:Texp}, the coefficient $s_i=\partial_{u_{\sigma_i}}\partial_{v_{\sigma_i}} g_i(u_i,v_i)|_{\gamma}$ becomes dependent on $s_{i-1}, q_i$ and $q_{i-1}$ and therefore there is no additional basis function associated to the edge $\tau_i$. 

Applying the previous construction to every $\gamma\in\cM_0$, we obtain a collection of splines $\cB_0(\gamma)\subseteq\cS_k^1(\cM_\gamma)$ for each $\gamma\in\cM_0$. We lift the splines $f\in\cS_k^1(\cM_\gamma)$ to functions on $\cM$ by defining $f_\sigma=0$ for every $\sigma\notin\cM_\gamma$. To simplify the exposition, we abuse the notation, and will also call $f$ the lifted spline on $\cM$, and $\cB_0(\gamma)$ the collection of those splines. 

\begin{df}\label{def:B0}
For a topological surface $\cM$, let $\cB_0 \subseteq S_k^1(\cM)$ be the set of linearly independent functions defined by
\begin{equation}\label{eq:B0}
  \cB_0=\bigcup_{\gamma\in\cM_0} \cB_{0}(\gamma),
\end{equation}
where $\cB_0(\gamma)\subseteq \cS_k^1(\cM_\gamma)$, for each vertex $\gamma\in\cM$. 
\end{df}

By construction, the collection of splines in $\cB_0(\gamma)$, for each vertex $\gamma\in\cM_0$, and $\cB_0$, are linearly independent. Moreover, the number of elements in $\cB_0$ coincides with the dimension of $\cH_k(\cM)$ and hence they constitute a basis for the spline space $\cS_k^1(\cM)$ whose Taylor map $T$ \eqref{eq:T} is not zero.\section{Splines on a face}

Let $\cF_{k}(\cM)$ be the spline functions in $\cS_{k}^{r}(\cM)$ with vanishing Taylor expansion along all the edges of $\cM$, that is, $\cF_{k}(\cM)= \cS_{k}^{r}(\cM) \cap \ker D$. 

An element $f$ is in  $\cF_{k}(\cM)$ if and only if  $c_{i,j}^{\sigma}(f)=0$ for $i  \le 1$ or $i \ge m-1$,
$ j \le 1$ or $ j \ge  m-1$ for all $\sigma\in \cM_{2}$.

Let $\cF_{k}(\sigma)$ be the elements in  $\cF_{k}(\cM)$ with $c_{i,j}^{\sigma'}(f)=0$ for $0\leq i,j \leq m$ and $\sigma'\neq \sigma$.
\begin{itemize}
 \item The dimension of $\cF_{k}(\sigma)$ is $(2\, k-r-3)_{+}^{2}$.
 \item A basis of $\cF_{k}(\sigma)$ is $N_{i}(u_{\sigma})N_{j}(v_{\sigma})$ for $1< i, j < m-1$.
\end{itemize}

We easily check that $\cF_{k}(\cM)= \oplus_{\sigma} \cF_{k}(\sigma)$, which implies the following result:
\begin{lm} \label{lm:dim:F}
The dimension of $\cF_{k}(\cM)$ is $(2 k-r-3)_{+}^{2} F_{2}$, where $F_{2}$ is the number of (quadrangular) faces of $\cM$.
\end{lm}

\paragraph{Basis functions associated to a face.}
The set $\cF_{k}(\cM)$ of basis functions associated to faces is obtained by taking the union of the bases of $\cF_{k}(\sigma)$ for all faces $\sigma \in \cM_2$, that is,
\begin{equation}\label{eq:B2} 
\cB_{2} := \{ N_{i}(u_{\sigma}) N_{j}(v_{\sigma}), 1< i, j < m-1, \sigma\in \cM_{2} \}.
\end{equation}
 \section{Dimension and basis of Splines on $\cM$}\label{sec:dim}
We have now all the ingredients to determine the dimension of $\cS_{k}^{1,r}(\cM)$ and a basis.

\begin{theorem}\label{thm:dim}
Let $\ms^{*} = \max \{\ms(\tau)\mid \tau\in\cM_{1}\}$. Then, for $k\ge \ms^{*},$ 
\[ 
\begin{array}{rclll}
	\dim \cS_{k}^{1} (\cM) & = & (2k-r-3)^{2} F_{2}  + \sum_{\tau\in \cM_{1}} \tilde{d}_{\tau}(k,r)  + 4 F_{2} -9 F_{1} + 3 F_{0} + F_{+} 
\end{array}
\]
where
\begin{itemize}
   \item $\tilde{d}_{\tau}(k)$ is the dimension of the syzygies of the gluing data along $\tau$ in degree $\le k$,
   \item $F_2$ is the number of rectangular faces, 
  \item $F_1$ is the number of edges,    
  \item $F_0$ (resp. $F_{+}$) is the number of (resp. crossing) vertices,
\end{itemize}
\end{theorem}
\begin{proof}
By construction,  $\cK_{k}(\cM)=\cS_{k}^{1,r}(\cM) \cap \ker T$ is the set of splines in $\cS_{k}^{1,r}(\cM)$, which Taylor expansion at all the vertices vanish and $\cH_{k}(\cM)$ is the image of $\cS_{k}^{1,r}(\cM)$ by the Taylor map $T$. Thus we have the following exact sequence:
\begin{equation}\label{eq:Texact}
  0\rightarrow \cK_{k}(\cM) \rightarrow \cS_{k}^{1,r}(\cM)\stackrel{T}{\longrightarrow} \cH_{k}(\cM) \rightarrow 0.
 \end{equation}
 
By construction, $\cE_{k}(\cM)$ is the set of splines in $\cK_{k}(\cM)$ with a support along the edges of $\cM$, so that $D(\cK_{k}(\cM)) = \cE_{k}(\cM)$.
The kernel of $D: \oplus_{\sigma} \cR_{k}(\cM) \rightarrow  \oplus_{\sigma} \cR_{k}(\cM) $ is $\cF_{k}(\cM)$. As $\cF_{k}(\cM)\subset \cK_{k}(\cM)$, we have the exact sequence
\begin{equation}\label{eq:Dexact}
0\rightarrow \cF_{k}(\cM) \rightarrow \cK_{k}(\cM) \stackrel{D}{\longrightarrow} \cE_{k}(\cM) \rightarrow 0. 
\end{equation}
 From the exact sequences \eqref{eq:Texact} and \eqref{eq:Dexact}, we have
\begin{eqnarray*}
  \dim  \cS^{1,r}_{k}(\cM)&= & \dim \cH_{k}(\cM)+ \dim \cK_{k}(\cM) \\
   & =& \dim \cH_{k}(\cM)+ \dim \cE_{k}(\cM) + \dim \cF_{k}(\cM) 
\end{eqnarray*}
We deduce the dimension formula using Lemma \ref{lm:dim:E}, Proposition \ref{prop:dim_H} and Lemma \ref{lm:dim:F}, as in \cite[proof of Theorem 6.3]{mourrain_dimension_2016}.
 \end{proof}

 \paragraph{Basis of $\cS_{k}^{1,r}(\cM)$.} A basis of $\cS_{k}^{1,r}(\cM)$ is obtained by taking
 \begin{itemize}
  \item the basis $\cB_{0}$ of $\cV_{k}(\cM)$ attached to the vertices of $\cM$ and defined in \eqref{eq:B0},
  \item the basis $\cB_{1}$ of $\cE_{k}(\cM)$ attached to the edges of $\cM$ and defined in \eqref{eq:B1},
 \item the basis $\cB_{2}$ of $\cF_{k}(\cM)$ attached to the faces of $\cM$ and defined in \eqref{eq:B2}.
 \end{itemize}
 
\section{Examples}
To illustrate the construction, we detail an example of a simple mesh, where a point of valence $3$ is connected to a crossing point. The construction can be extended to points of arbitrary valencies, in a more complex mesh.

We consider the mesh $\cM$ composed of $3$ rectangles $\sigma_{1},\sigma_{2},\sigma_{3}$ glued
around an interior vertex $\gamma$, along the $3$ interior edges $\tau_{1},
\tau_{2}, \tau_{3}$. There are $6$ boundary edges and $6$ boundary vertices
$\delta_{1},\delta_{2},\delta_{3}$, $\epsilon_{1},\epsilon_{2},\epsilon_{3}$.
\begin{figure}[ht]
	\begin{center}
          \includegraphics[height=4cm]{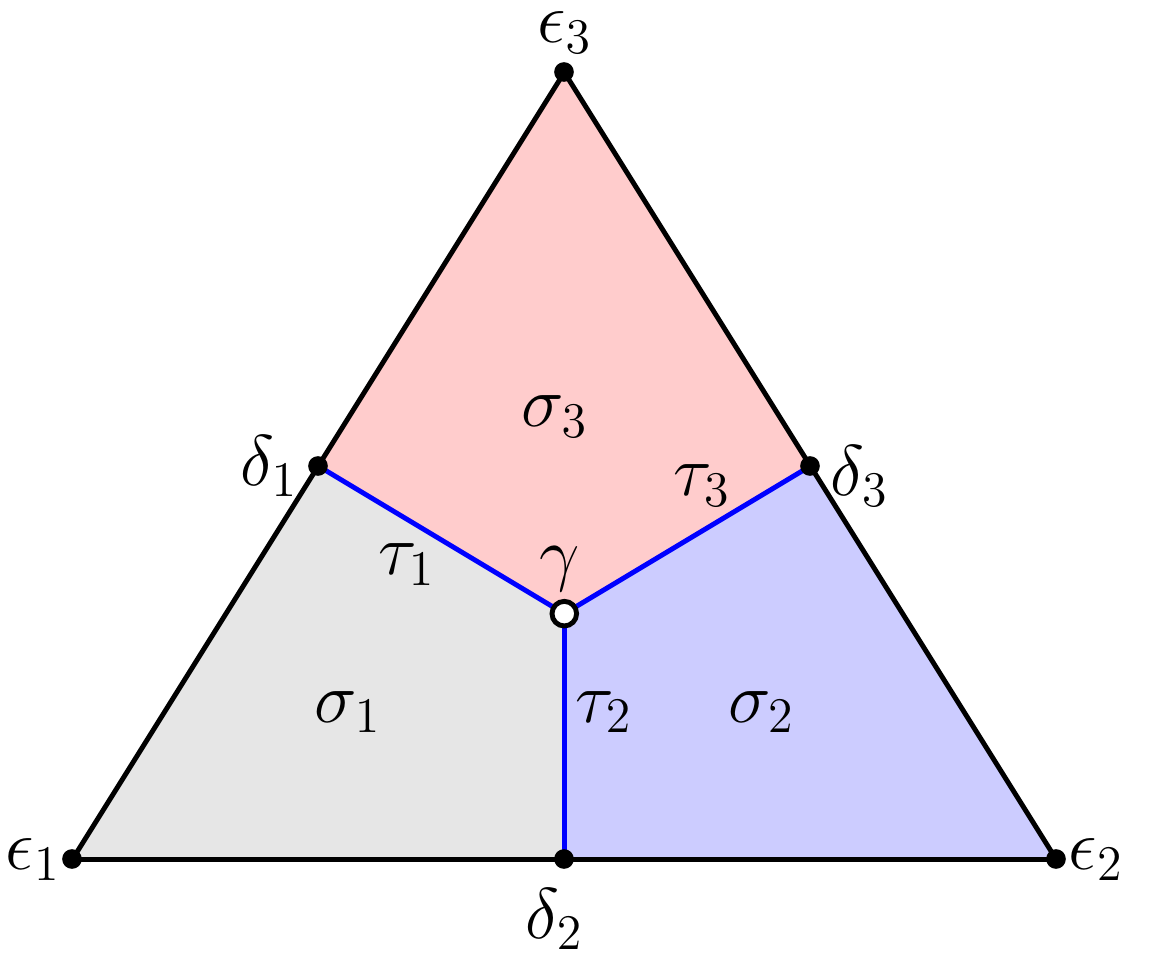}
	\end{center} 
	\vspace{-0.5cm}
	\caption{Smooth corner.}\label{fig:roundcorner}
      \end{figure}
We use the symmetric glueing corresponding to the angle $\frac{2\pi}{3}$   at $\gamma$  and 
$\frac{\pi}{2}$ at $\delta_{1},\delta_{2},\delta_{3}$. 

We choose the gluing data $[a,b,c]$ along an edge $\tau_{i}$ given by Formula \eqref{eq:construct2}: 
$$ 
\begin{array}{rcl}
a(u) &=& \md_{0}(u) \\ 
b(u) &=& - \md_{0}(u) - \md_{1}(u)\\ 
c(u) &=& \md_{0}(u) + \md_{1}(u)	
\end{array}
$$
where $\md_{0}=\tilde{N}_{0}(u)+\tilde{N}_{1}(u)$, $\md_{1}= \tilde{N}_{2}(u)+\tilde{N}_{3}(u)+\tilde{N}_{4}(u)$ for the b-spline basis $\tilde{N}_{0},\ldots,\tilde{N}_{5}$ of $\cU^{0}_{2}$ and where $u=0$ corresponds to $\gamma$.
This gives
$$ 
a(u)= \left\{
\begin{array}{ll}
  -1+4\,{u}^{2} & 0\leq u\leq \frac{1}{2}\\
  \noalign{\medskip}0 & \frac{1}{2}\leq u \leq 1\\ 
\end{array}
\right.,
b(u)= -1,\quad
c(u)=1;
$$
The degrees of the $\mu$-bases of the different components are respectively $\mu_{1}=0,\nu_{1}=2$,
$\mu_{2}=0$, $\nu_{2}=0$. Thus the separability  is reached from the degree $k\geq 4$.

We are going to analyze the spline space $\cS^{1,1}_{4}(\cM)$ for specific gluing data.
An element $f\in \cS^{1,1}_{4}(\cM)$ is represented on each cell $\sigma_{i}$ ($i=1,2,3$) by a tensor product b-spline of class $C^{1}$ with $8\times 8$ b-spline coefficients:
$$ 
f_{k} := \sum_{0\le i,j\le 7} c_{i,j}^{k}(f) N_{i,j}(u_{k},v_{k}),
$$
where $N_{i,j}(u,v)=N_{i}(u) N_{j}(v)$ and $\{N_{0}(u),\ldots, N_{7}(u)\}$ is the basis of $\cU^{1}_{k}$.
We describe an element $f\in \cS^{1,1}_{4}(\cM)$ as a triple of b-spline functions
\[\biggl[\, \sum_{0\le i,j\le 7} c_{i,j}^{1} N_{i,j}\,,\,  \sum_{0\le i,j\le 7} c_{i,j}^{2} N_{i,j}\,,\,  \sum_{0\le i,j\le 7} c_{i,j}^{3}  N_{i,j}\biggr].\]
The separability is reached at degree $4$ and we have the following basis elements, described by a triple of functions which are decomposed in the b-spline bases of each face: 

\noindent $\bullet$ The number of basis functions attached to $\gamma$ is $6= 1+2+3$.

\noindent \quad -- The basis function associated to the value at $\gamma$ is
{ \small \begin{align*}
\qquad \biggl[&N_{{0,0}}+\frac{1}{3}\,N_{{0,2}}+N_{{0,3}}+N_{{0,4}}+2\,N_{{1,3}}+2\,N_{{1,4}
  }+\frac{1}{3}\,N_{{2,0}}+N_{{3,0}}+N_{{4,0}},\\
&\; N_{{0,0}}+\frac{1}{3}\,N_{{2,0}}+N_{{3,0}
}+N_{{4,0}}+3\,N_{{0,1}}+{\frac {31}{3}}\,N_{{0,2}}+17\,N_{{0,3}}+17\,
N_{{0,4}}\\
& \hspace{6.3cm}+14\,N_{{1,2}}+34\,N_{{1,3}}+34\,N_{{1,4}},\\
&\; N_{{0,0}}+3\,N_{{1
,0}}+{\frac {31}{3}}\,N_{{2,0}}+17\,N_{{3,0}}+17\,N_{{4,0}} +\frac{1}{3}\,N_{{0
  ,2}}+N_{{0,3}}+N_{{0,4}}\\
& \hspace{8.5cm}+2\,N_{{1,3}}+2\,N_{{1,4}}\biggr].
\end{align*}}
\noindent{}\quad -- The two basis functions associated to the derivatives at $\gamma$ are
{\small\begin{align*}
\qquad \biggl[&N_{{0,1}}+\frac{10}{3}\,N_{{0,2}}+\frac{16}{3}\,N_{{0,3}}+\frac{16}{3}\,N_{{0,4}}
+\frac{14}{3}\,N_{{1,2}}+{\frac {32}{3}}\,N_{{1,3}}+{\frac {32}{3}}\,N_{{1,4}}\,,\\
&\; N_{{1,0}}+\frac{10}{3}\,N_{{2,0}}+\frac{16}{3}\,N_{{3,0}}+\frac{16}{3}\,N_{{4,0}}\,,\\
&-N_{{0,1}}-\frac{10}{3}\,N_{{0,
2}}-\frac{16}{3}\,N_{{0,3}}-\frac{16}{3}\,N_{{0,4}}-\frac{16}{3}\,N_{{1,2}}-{\frac {32}{3}}\,N
_{{1,3}}-{\frac {32}{3}}\,N_{{1,4}}\\
&\hspace{5.2cm}-N_{{1,0}}-\frac{10}{3}\,N_{{2,0}}-\frac{16}{3}\,N
_{{3,0}}-\frac{16}{3}\,N_{{4,0}}\biggr]\,,\\
& \biggl [N_{{1,0}}+\frac{10}{3}\,N_{{2,0}}+\frac{16}{3}\,N_{{3,0}}+\frac{16}{3}\,N_{{4,0}}\,,\\
&-N_{{0,1}}-
\frac{10}{3}\,N_{{0,2}}-\frac{16}{3}\,N_{{0,3}}-\frac{16}{3}\,N_{{0,4}}-\frac{14}{3}\,N_{{1,2}}-{
                  \frac {32}{3}}\,N_{{1,3}}-{\frac {32}{3}}\,N_{{1,4}}\,,\\
&              
-N_{{1,0}}-\frac{10}{3}\,N_{{2,0}}-\frac{16}{3}\,N_{{3,0}}-\frac{16}{3}\,N_{{4,0}}+N_{{0,1}}+\frac{10}{3}\,N_{{0,2}}+\frac{16}{3}\,N_{{0,3}}\\
&\hspace{4.6cm}+\frac{16}{3}\,N_{{0,4}}+\frac{14}{3}\,N_{{1,2}}+{\frac {32}{3}}\,N_{{1,3 
}}+{\frac {32}{3}}\,N_{{1,4}}\biggr].
\end{align*}}

\noindent{}\quad -- The three basis functions associated to the cross derivatives at $\gamma$ are 
{\small
\begin{align*}         
\qquad \biggl[ & -\frac{4}{3}\,N_{{0,2}}-\frac{8}{3}\,N_{{0,3}}-\frac{8}{3}\,N_{{0,4}}+N_{{1,1}}-\frac{4}{3}\,N_{{1,2}
  }-\frac{16}{3}\,N_{{1,3}}-\frac{16}{3}\,N_{{1,4}}\,,\\
 & \hspace{5.5cm}\; -\frac{4}{3}\,N_{{2,0}}-\frac{8}{3}\,N_{{3,0}}-\frac{8}{3}\,
N_{{4,0}}\,,\; 0 \, \biggr]\,,\\
\biggl[ & -\frac{4}{3}\,N_{{2,0}}-\frac{8}{3}\,N_{{3,0}}-\frac{8}{3}\,N_{{4,0}}\,,\\
&-\frac{4}{3}\,N_{{0,2}}-\frac{8}{3}\,N_
  {{0,3}}-\frac{8}{3}\,N_{{0,4}}-\frac{4}{3}\,N_{{1,2}}-\frac{16}{3}\,N_{{1,3}}-\frac{16}{3}\,N_{{1,4}}\,,\\
&
-\frac{4}{3}\,N_{{2,0}}-\frac{8}{3}\,N_{{3,0}}-\frac{8}{3}\,N_{{4,0}}+N_{{1,1}}-\frac{4}{3}\,N_{{0,2}}
-\frac{8}{3}\,N_{{0,3}}-\frac{8}{3}\,N_{{0,4}}\\
&\hspace{6cm}-\frac{4}{3}\,N_{{1,2}}-\frac{16}{3}\,N_{{1,3}}-\frac{16}{3}\,N_
  {{1,4}}\biggr],\\
&\biggl[-\frac{4}{3}\,N_{{2,0}}-\frac{8}{3}\,N_{{3,0}}-\frac{8}{3}\,N_{{4,0}}\,,\; 0\,,\\
&\hspace{1.5cm}-\frac{4}{3}\,N_{{0,2}}-\frac{8}{3}\,
N_{{0,3}}-\frac{8}{3}\,N_{{0,4}}-\frac{4}{3}\,N_{{1,2}}-\frac{16}{3}\,N_{{1,3}}-\frac{16}{3}\,N_{{1,4}
  }\biggr].
\end{align*}}
\noindent $\bullet$ There are $4= 1+2+2-1$ basis functions attached to $\delta_{i}$:
{\small \begin{align*}
\scalebox{1.3}{[}N_{{0,7}}\,, \; N_{{7,0}}+2\,N_{{7,1}}\,,\; 0\scalebox{1.3}{]},
\scalebox{1.3}{[}N_{{0,6}}\,,\; N_{{6,0}}+2\,N_{{6,1}}\,,\; 0\, \scalebox{1.3}{]},
\scalebox{1.3}{[}N_{{1,7}}\,,\; -N_{{7,1}}\,,0\,\scalebox{1.3}{]},
\scalebox{1.3}{[}N_{{1,6}}\,,-N_{{6,1}}\,,\; 0\,\scalebox{1.3}{]}.
\end{align*}}
The basis functions associated to the other boundary points $\delta_{2},\delta_{3}$ are obtained by cyclic permutation.

$\bullet$ There are $5=14-5-4$ basis functions attached to edge $\tau_{1}$:
{\small \begin{align*} 
&\scalebox{1.3}{[}-N_{{1,2}}\,,\; N_{{2,1}}\,,\;0\, \scalebox{1.3}{]},
  \scalebox{1.3}{[}-N_{{1,3}}\,,\;N_{{3,1}}\,,\;0\, \scalebox{1.3}{]},
  \scalebox{1.3}{[}-N_{{1,4}}\,,\;N_{{4,1}}\,,\;0\, \scalebox{1.3}{]},\\
&\scalebox{1.3}{[}-N_{{1,5}}\,,\;N_{{5,1}}\,,\;0\, \scalebox{1.3}{]},\scalebox{1.3}{[}N_{{0,5}}+2\,N_{{1,5}}\,,\;N_{{5,0}}\,,\;0\, \scalebox{1.3}{]}.
\end{align*}}
The basis functions associated to the other edges $\tau_{2},\tau_{3}$ are obtained by cyclic permutation.

$\bullet$ For the remaining boundary points, boundary edges and faces, we have the following $36\times 3$ basis functions 
{\small \begin{align*}
\scalebox{1.3}{[}N_{i,j}\,,\;0\,,\;0\, \scalebox{1.3}{]},
\scalebox{1.3}{[}0\,,\;N_{i,j}\,,\;0\, \scalebox{1.3}{]},
\scalebox{1.3}{[}0\,,\;0\,,\;N_{i,j}\, \scalebox{1.3}{]}, \quad \mathrm{for}\ \ 2\leq i,j\leq 7.
\end{align*}}
The dimension of the space $\cS^{1,1}_{4}(\cM)$ is $6+ 3\times (4+ 5 +36)= 141$.

A similar construction applies for an edge of a general mesh connecting an interior vertex $\gamma$ of any valency $\neq 4$ to another vertex $\gamma'$.
If $\gamma'$ is a crossing vertex, the numbers of basis functions attached to the vertices and the edge do not change.
If $\gamma'$ is not a crossing vertex, the number of basis functions attached to the non-crossing vertex $\gamma'$ becomes $5$ and there are $4$ basis functions attached to the edges.
In the case, where the edge connects two crossing vertices, there are $4$ basis functions attached to each crossing vertex and $8$ basis functions attached to the edge.

The glueing data used in this construction require a degree 4 for the separability.
For the mesh of Figure \ref{fig:roundcorner}, it is possible to use linear glueing data and bi-cubic b-spline patches. The dimension of bi-cubic $G^{1}$ splines with the linear glueing data is $72$. Depending on topology of the mesh, it is possible to construct and the choice of the glueing data, it is possible to use low degree b-spline patches for the construction of $G^{1}$ splines. In Figure \ref{fig:bicubic}, examples of $G^{1}$ bicubic spline surfaces are shown, for meshes with valencies at most $3, 4$ and $6$. The $G^{1}$ surface is obtained by least-square projection of a  $G^{0}$ spline onto the space of $G^{1}$ splines. 
\begin{figure}[ht]
  \begin{center}
    {\includegraphics[height=4cm,width=4cm]{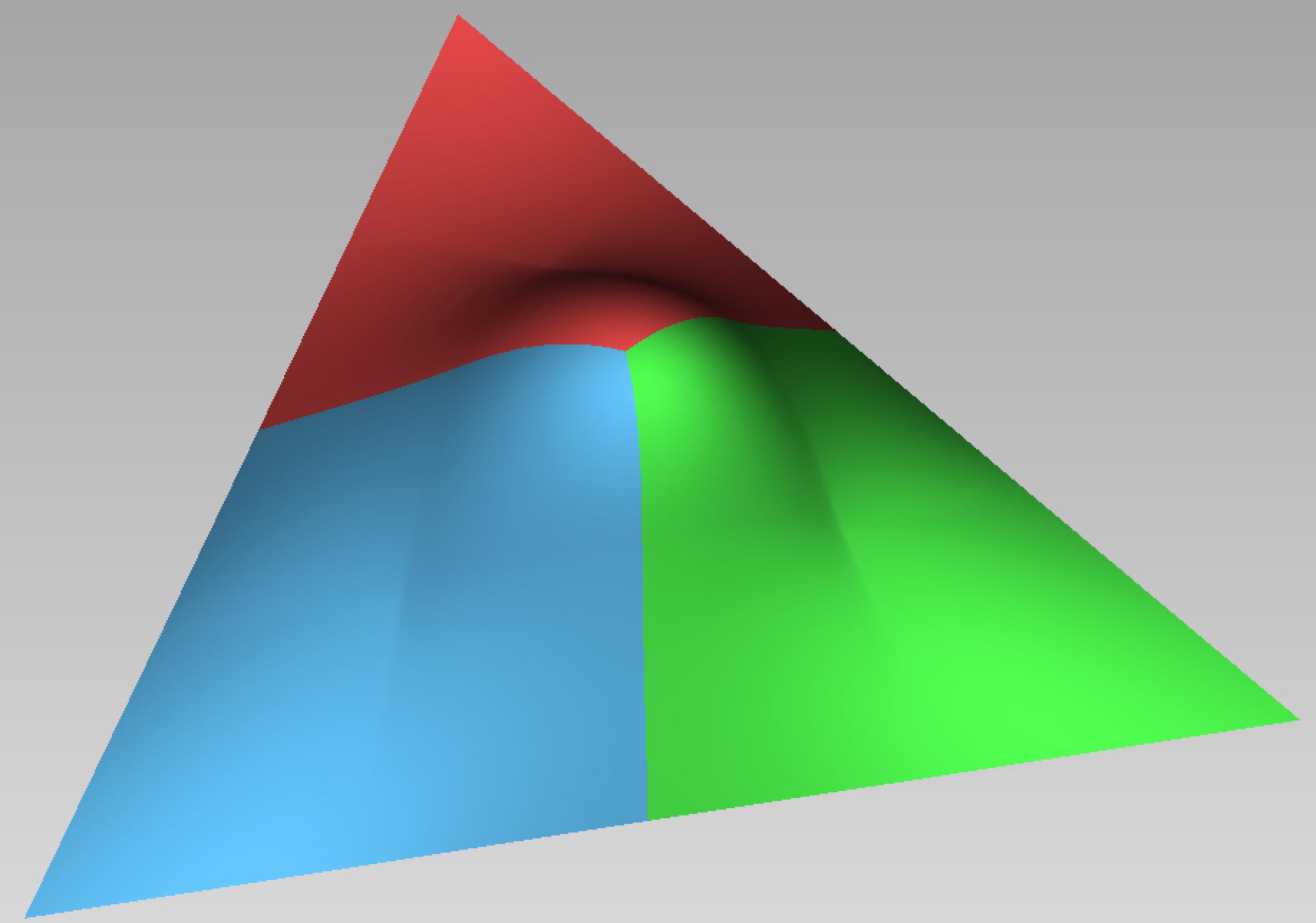}}
    {\includegraphics[height=4cm]{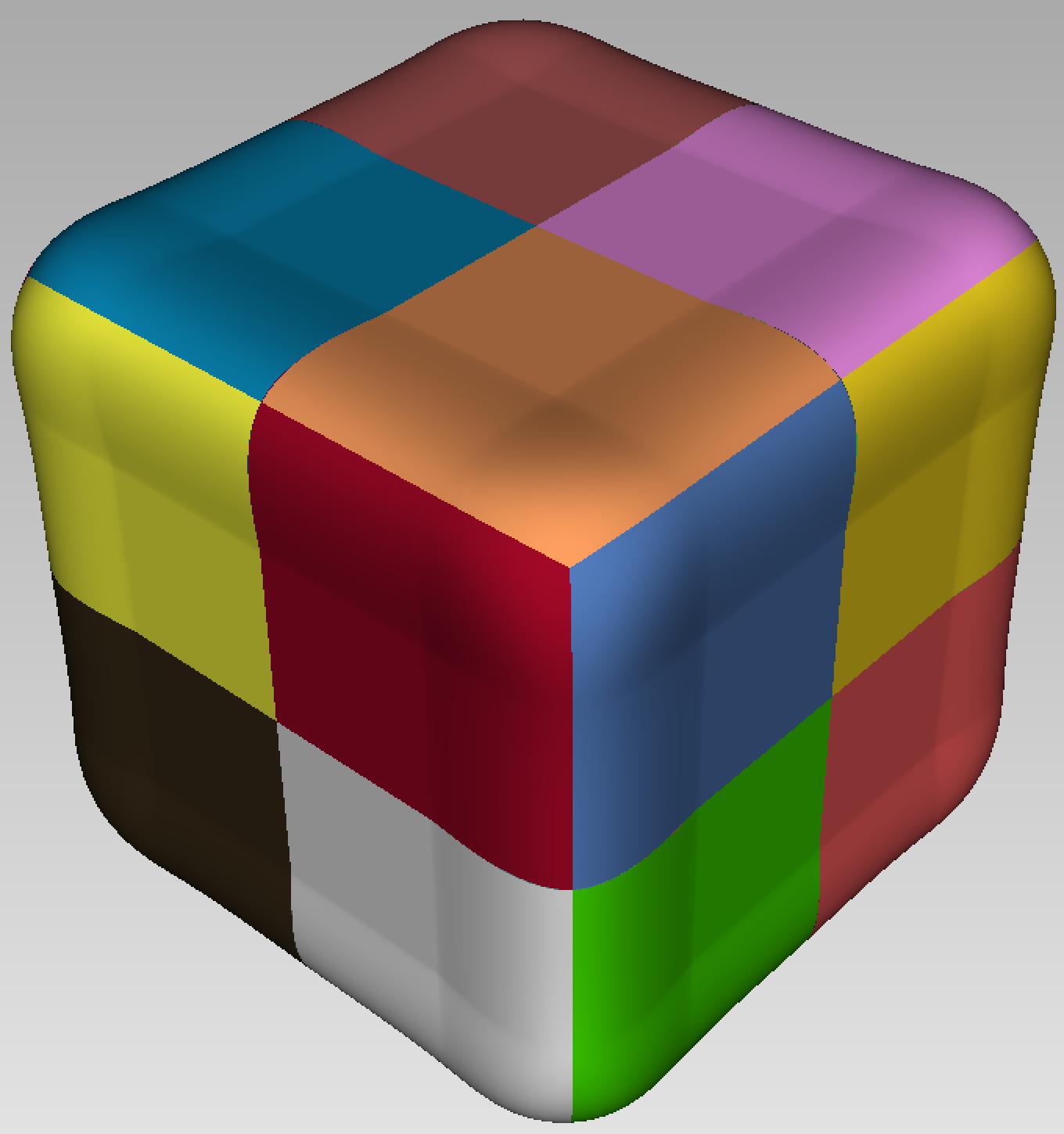}}
    {\includegraphics[height=4cm]{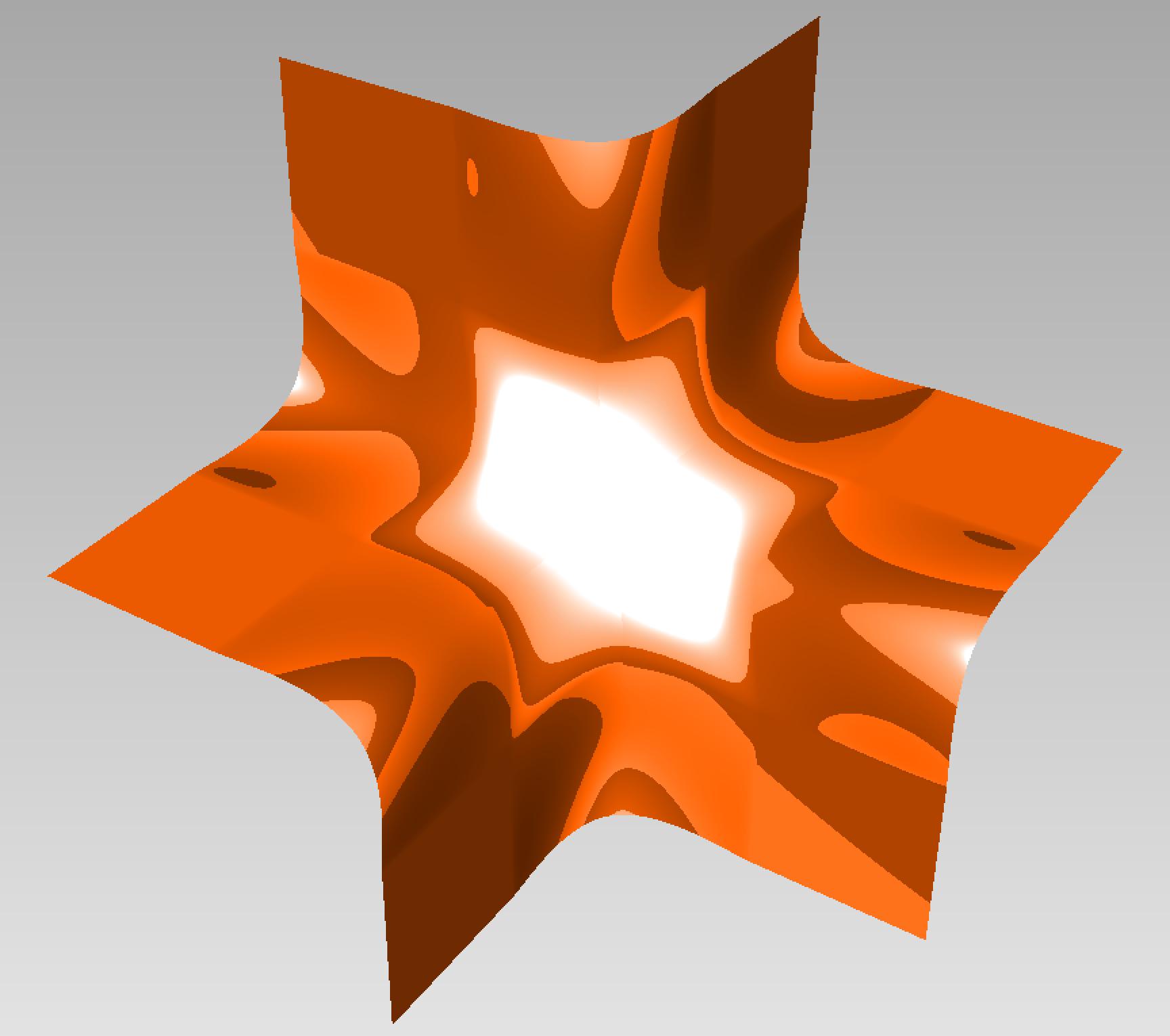}}
\end{center}
\caption{Examples of bi-cubic $G^{1}$ surfaces}\label{fig:bicubic}
\end{figure}

\section*{Concluding remarks}
We have studied the set of smooth b-spline functions defined on quadrilateral meshes of arbitrary topology, with 4-split macro-patch elements. Our study has focused on determining the dimension of the space of geometrically continuous $G^1$ splines of bounded degree. We have  provided a construction for the basis of the space  composed of tensor product b-spline functions. We have also illustrated our results with examples concerning parametric surface construction for simple topological surfaces.
Further extensions include the explicit construction of transition maps which ensure that the differentiability conditions are fulfilled, and the study of spline spaces with different macro-patch elements leading to a lower degree of the basis functions, the analysis of the numerical conditioning of the representation of the  $G^{1}$-splines in the chosen basis,  the use of these basis functions for approximation, in particular, in fitting problems and in iso-geometric analysis.

\paragraph{Acknowledgements:} The work is partial supported by the Marie Sklodowska-Curie  Innovative Training Network ARCADES (grant agreement No 675789) from the European Union's
Horizon 2020 research and innovation programme. 


\newcommand{\etalchar}[1]{$^{#1}$}

\end{document}